\documentclass[preprint,12pt,sort&compress]{elsarticle}

\usepackage{amssymb}
\usepackage{amsmath}

\usepackage{amsmath,graphicx,amsfonts,amsthm,amssymb,setspace,mathtools,float,tikz,caption,fancyhdr,hyperref,bbm,scalerel,stackengine}
\stackMath
\journal{Journal of Functional Analysis}

\usetikzlibrary{calc}

\newtheorem{lem}{Lemma}[section]
\newtheorem{prop}[lem]{Proposition}
\newtheorem{thm}[lem]{Theorem}
\newtheorem{cor}[lem]{Corollary}

\theoremstyle{definition}
\newtheorem{defi}[lem]{Definition}
\newtheorem{ex}[lem]{Example}
\newtheorem{asm}[lem]{Assumption}
\theoremstyle{remark}
\newtheorem{rem}[lem]{Remark}

\newcommand{\ol}[1]{\overline{#1}}
\newcommand{\ul}[1]{\underline{#1}}
\newcommand{\Mc}[1]{\mathcal{#1}}
\newcommand{\Mb}[1]{\mathbb{#1}}
\newcommand{\Mf}[1]{\mathfrak{#1}}
\newcommand{\Mr}[1]{\mathrm{#1}}
\renewcommand{\Cup}{\bigcup}

\newcommand{\core}{\Mc{F}\cap C_c(X)}
\newcommand{\deco}[1]{#1
}

\newcommand{\supp}{\mathrm{supp}}

\newcommand{\sgn}{\mathrm{sgn}}
\newcommand{\zero}{\boldsymbol{0}}
\newcommand{\indi}{\mathbbm{1}}

\newcommand{\der}{\mathrm{d}}

\newcommand{\MR}[1]{\url{https://mathscinet.ams.org/mathscinet-getitem?mr=#1}}

\newcommand{\same}{\ul{\hspace{1.5cm}}}

\DeclareMathOperator{\LIP}{LIP}
\DeclareMathOperator{\Lip}{Lip}

\newcommand{\DeclareDescCond}[3]{%
	\newcounter{#1}%
	\expandafter\renewcommand\csname the#1\endcsname{(#3)}%
	\expandafter\newcommand\csname #2\endcsname[1]{%
		\refstepcounter{#1}%
		\item[\csname the#1\endcsname]\label{##1}%
	}%
}
\DeclareDescCond{energycond}{Econd}{E\arabic{energycond}}
\DeclareDescCond{energydashcond}{Edashcond}{E\arabic{energydashcond}'}
\DeclareDescCond{measurecond}{Mcond}{M\arabic{measurecond}}
\DeclareDescCond{normlemma}{Ncond}{N\arabic{normlemma}}
\DeclareDescCond{domimeas}{DMcond}{DM\arabic{domimeas}}
\DeclareDescCond{gammacond}{Gcond}{G\arabic{gammacond}}
\DeclareDescCond{bargammacond}{BGcond}{BG\arabic{bargammacond}}
\begin{document}

\begin{frontmatter}



\title{\texorpdfstring{Construction of $p$-energy measures associated with strongly local $p$-energy forms}{Construction of p-energy measures associated with strongly local p-energy forms}}


\author{\texorpdfstring{K\^ohei Sasaya\corref{cor}\fnref{jsps}}{K\^ohei Sasaya}} 
\ead{sasaya@g.ecc.u-tokyo.ac.jp}
\ead[url]{https://koheisasaya.github.io}
\fntext[jsps]{JSPS Research Fellow (PD)}
\cortext[cor]{Corresponding author}
\affiliation{organization={Graduate School of Mathematical Sciences, The University of Tokyo},
            addressline={3-8-1, Komaba}, 
            city={Meguro},
            postcode={153-8914}, 
            state={Tokyo},
            country={Japan}}
\begin{abstract}
We construct canonical $p$-energy measures associated with strongly local $p$-energy forms without assuming self-similarity. Here, $p$-energy forms are $L^p$-analogues of Dirichlet forms, which have recently been studied mainly on fractals. Furthermore, we prove that these measures satisfy the chain and Leibniz rules, and that such ``good'' energy measures are unique. A key ingredient is a $p$-energy analogue of Le~Jan's domination principle. Moreover, we show that the Korevaar–Schoen-type $p$-energy measures defined by Alonso-Ruiz and Baudoin~\cite{AB} coincide with our canonical $p$-energy measures.
\end{abstract}	


%

\begin{keyword}
energy measure \sep $p$-energy form \sep Korevaar--Schoen-type $p$-energy \sep fractal \sep Clarkson's inequality

\MSC 28A80 \sep 31C25 \sep 31C45 \sep 31E05 \sep 46E36

\end{keyword}

\end{frontmatter}

\section{Introduction}
Since the late 20th century, Sobolev spaces on non-smooth metric spaces have been studied extensively. One successful approach is based on ``upper gradients'' (see, e.g., \cite{AGS14a,Che,Sha}; see also \cite{Haj,HK}). This approach has also been used in recent work in differential geometry on non-smooth metric spaces (e.g., \cite{AGS14b}). 
Another approach is to construct Sobolev spaces on fractals using Dirichlet form theory. On many fractals, however, the upper-gradient approach may fail to yield a non-trivial first-order Sobolev space (for instance, on the standard Sierpi\'nski carpet (Figure \ref{figSC}) the Newtonian space collapses to $L^p$; see Remark~\ref{remthmunique}~(3).
\begin{figure}
\centering
\begin{minipage}{.49\linewidth}
\centering
\includegraphics[width=4.8cm]{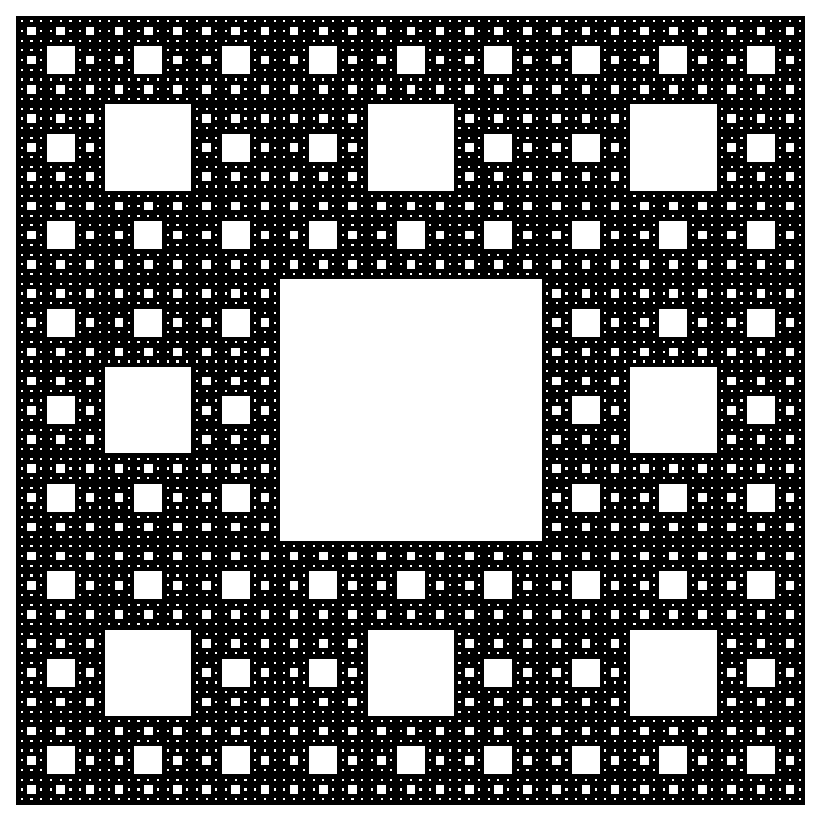}
\caption{Standard Sierpi\'nski Carpet} \label{figSC}
\end{minipage}
\begin{minipage}{.49\linewidth}
\centering
\includegraphics[width=5cm]{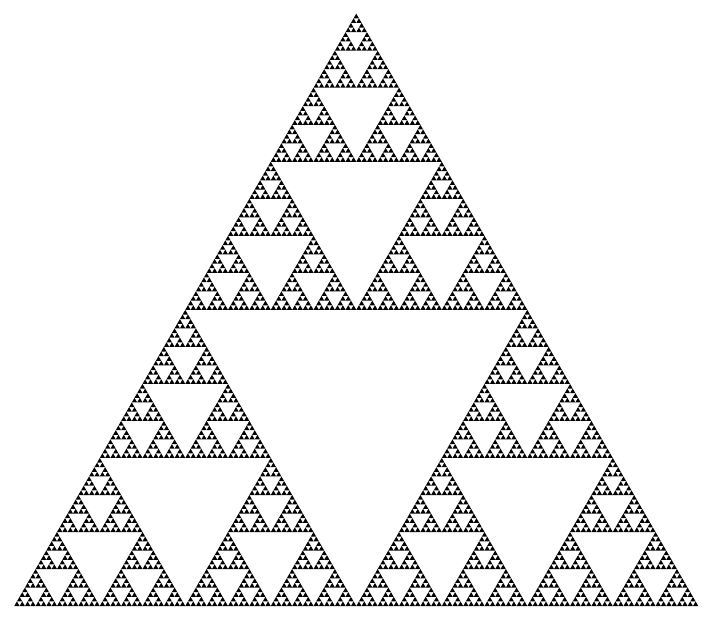}
\caption{Sierpi\'nski Gasket} \label{figSG}
\end{minipage}
\end{figure}
\par 

The study of analysis on fractals began in the 1980s. Inspired by Kesten's work~\cite{Kes86a,Kes86b}, Goldstein~\cite{Gol} and Kusuoka~\cite{Kus1} independently constructed Brownian motion on the Sierpi\'nski gasket (see Figure \ref{figSG}) as a scaling limit of random walks on graphs approximating the gasket. In the seminal paper~\cite{BP}, Barlow and Perkins provided detailed estimates of the transition density of this Brownian motion. Following these studies, Fukushima and Shima~\cite{FS} constructed, in 1992, the Dirichlet form on the gasket corresponding to this Brownian motion; this Dirichlet form can be viewed as a counterpart of a Sobolev space. (See also the paper~\cite{Kig89}, which directly constructs the associated Laplacian.) The scope of the subject is not limited to the gasket: important examples also include infinitely ramified self-similar sets such as the standard Sierpi\'nski carpet. Many studies have been devoted to constructing Dirichlet forms on fractals (see, e.g., \cite{HaK,Kaj21AG,Kig93,Kig12,Kig23,KZ}), the corresponding Markov processes (see, e.g., \cite{BB89,BB99,Ham,Lin}), and their properties (see, e.g., \cite{BBKT,BK,KL,Hin17,Kum}; see also \cite{Bar,Kig01,Str}). 
\par

Energy measures have also been studied in analysis on fractals. Let $(\Mc{E},\Mc{F})$ be a regular Dirichlet form on $L^2(X,d,\Mf{m})$. For $f\in\core$, the \emph{energy measure of $f$ associated with $(\Mc{E},\Mc{F})$}, denoted by $\mu_{\langle f\rangle}$, is defined as the unique Radon measure satisfying
\begin{equation}\label{defifunct4em}
	\int_{X} g\der\mu_{\langle f\rangle}=\Mc{E}(f,fg)-\frac{1}{2}\Mc{E}(f^2,g)\quad (\text{for any }g\in \core).
\end{equation}
Here $C_c(X)$ denotes the set of continuous functions with compact support. $\mu_{\langle f\rangle}$ is well-defined by the Riesz--Markov--Kakutani representation theorem, and it extends canonically to a general $f\in\Mc{F}$.
In smooth settings, $\der\mu_{\langle f\rangle}$ corresponds to $|\nabla f|^2\der x$. However, on fractals, $\mu_{\langle f\rangle}$ may be mutually singular with respect to the reference measure. Energy measures not only enable us to localize $\Mc{E}(f,f)$ for local regular Dirichlet forms, but are also closely related to the question of whether the Laplacian and the associated Markov process behave as in the smooth case after changing the reference measure. (See, e.g., \cite{Ram,Stu95,Stu96} for general settings and \cite{Kus2,Kig08,Kaj12,KoZ,KM} for fractals.) For further details in this direction, see the surveys \cite{Kaj13,Kaj21em}. 
\par

Recently, $L^p$-versions of Dirichlet forms on fractals have been considered in \cite{AES,Bau,CGQ,CGYZ,CheGYZ,GYZ,HPS,KSKS,KSdiff,KS25,KSx,Kig23,MS,Shi,Yan} to construct counterparts of $W^{1,p}$ on fractals. (One further motivation is an application to an attainment problem for dimensions of fractals arising in quasiconformal geometry; see \cite{MS}.) In this paper, we refer to such $L^p$-versions $(\Mc{E},\Mc{F})$ of Dirichlet forms as $p$-energy forms. A precise definition is given in Section \ref{secpre}. This naturally leads to the question of how to construct $p$-energy measures for such $p$-energy forms in an axiomatic setting.
\par

The main difficulty is that, for general $p$-energy forms, the usual functional identity \eqref{defifunct4em} has no direct $p$-counterpart that yields a usable positivity argument. In this paper, we construct canonical $p$-energy measures under basic axioms (without any self-similarity) and establish their uniqueness and fundamental calculus rules such as the chain and Leibniz rules.
\par

To explain this point, we recall the following observation. Under the assumption of the $p$-Clarkson inequalities (counterparts of the parallelogram law), although one can define $\Mc{E}(f;g)$ in \cite{KSdiff}, the map $\Mc{E}(\cdot;\cdot)$ is linear only in the second variable. Thus, it is difficult to obtain a positive measure from a functional identity corresponding to \eqref{defifunct4em}. In previous studies \cite{AES,KSdiff,KS25,KSx,MS,Shi}, ``$p$-energy measures'' were constructed when $X$ was a self-similar fractal, and their constructions depended heavily on the self-similarity of the underlying $p$-energy form. Hence, as mentioned in \cite[Section 6.3]{Kig23} and \cite[Problem 10.4]{MS}, it was open whether one could construct such measures without any self-similarity. Beyond the self-similar setting, only a few constructions are known: in \cite{CheGYZ}, Chen, Gao, Yu and Zhang constructed $p$-energy forms and associated energy measures on scale-irregular Vicsek sets without self-similarity, but their construction relies on the one-dimensional structure of the sets. In addition, $p$-energy measures associated with Korevaar--Schoen-type $p$-energy forms have been considered in \cite{AB,KSKS}. In \cite{AB}, the forms and energies on PI spaces (see Definition \ref{defPIsp}) were constructed via De Giorgi's $\Gamma$-convergence; however, some fundamental properties of these energy measures, such as the chain and Leibniz rules, remained unknown. In \cite{KSKS}, under the ``weak-monotonicity type estimate,'' Korevaar--Schoen-type $p$-energy forms and associated energy measures satisfying fundamental properties were constructed within another framework (see Section \ref{secKS} for details).
\par

In the present paper, we only impose assumptions on the $p$-energy forms and on the underlying metric measure spaces that correspond to those for strongly local, regular Dirichlet spaces (see Section \ref{secpre} for details). These assumptions cover most $p$-energy forms studied in the literature on fractals and PI spaces, and in particular the previously constructed $p$-energy measures in \cite{AES,AB,CheGYZ,KSdiff,KSKS,KS25,KSx,MS,Shi}, including the cases of the Sierpi\'nski gasket and the standard Sierpi\'nski carpet, fit into our framework. 
\par

Moreover, on a class of fractals including the Sierpi\'nski gasket, canonical $p$-energies constructed directly (typically via discrete approximations) exhibit phenomena such as mutual singularity of the $p$-energy measure of $f$ and the $q$-energy measure of $g$ for $p\neq q$ (see \cite{KSx,KS25,AES}). In contrast, if one constructs $p$-Sobolev spaces by the upper-gradient method (as in \cite{Sha}), or starts from a Dirichlet form $(E,F)$ and takes $\Mc{F}_{(p)}\subset F\cap L^p(\Mf{m})$ (as in \cite{BBR,Kuw}), then the associated energy measures are available in a straightforward way, and such a phenomenon does not occur.
\par

Our main result is Theorem \ref{thmunique}. A key ingredient is a domination principle for $p$-energy measures. Roughly speaking, it states that an inequality between energies implies the corresponding inequality for the associated measures. This principle was proved by Le~Jan~\cite{LeJ} when $p=2$. This principle is also of independent interest and may be useful in applications; see Theorem \ref{thmdomination}.
\par

By the uniqueness statement in Theorem \ref{thmunique}, our canonical $p$-energy measures agree with those constructed previously in \cite{AES,AB,CheGYZ,KSdiff,KSKS,KS25,KSx,MS,Shi}. Unfortunately, we do not obtain new examples of $p$-energy forms for which $p$-energy measures are constructed directly from our results, since the known $p$-energy forms are essentially included in the cases of self-similar fractals or Cheeger spaces. Nevertheless, our results are applicable to proving properties of known energy measures; one example is the uniqueness and the chain and Leibniz rules for the energy measures constructed in \cite{AB}.
\par
The remainder of this paper is organized as follows. Section \ref{secpre} is devoted to the statement of our main results. In Section \ref{secunique}, we prove the domination principle for $p$-energy measures. Sections \ref{secbas}--\ref{secchain} are devoted to the proof of the main theorem. In Section \ref{secbas}, we show some basic properties of $p$-energy forms. In Section \ref{secconst}, we construct $\mu_{\langle f\rangle}$ for $f\in\core$ such that \ref{condmeasene} holds. In Section \ref{secineqs}, we prove \ref{condmeasloc}--\ref{condmeasclk} and define the $p$-energy measures for general $f\in\Mc{F}$. In Section \ref{secchain}, we prove \ref{condmeaschain}, \ref{condmeasLeibniz} and some corollaries. Lastly, we review previous results on Korevaar--Schoen-type $p$-energy forms and apply our results to them in Section \ref{secKS}.
\subsection*{Notation}
In this paper, we use the following notation.
\begin{itemize}
	\item For $a,b\in\Mb{R}$, we write $a\vee b=\max\{a,b\}, a\wedge b=\min\{a,b\}$ and $a^+=a\vee 0$. We use the same operations for $\Mb{R}$-valued functions (e.g., $f\vee g=\max\{f,g\}$). 
	\item We define $\sgn:\Mb{R}\to\Mb{R}$ by $\sgn(0)=0$ and $\sgn(t)=t/|t|\ (t\ne0)$.
	\item $\indi_A$ denotes the indicator function of a set $A$.
	\item We write $\zero:=(0,\dots,0)\in\Mb{R}^n$ for $n\in\Mb{N}$.
\end{itemize}
Let $(X,d)$ be a metric space. We also use the following notation.
\begin{itemize}
	\item We denote the Borel $\sigma$-algebra of $(X,d)$ by $\Mc{B}(X,d)$ (or $\Mc{B}(X)$ if $d$ is clear). 
	\item $B_d(x,r)$ denotes the open ball $\{y\in X\mid d(x,y)<r\}$.
	\item For $A, B\subset X$, we write $d(A,B):=\inf\{d(x,y)\mid x\in A,\ y\in B\}$.
	\item $U\Subset V$ means that $\ol{U}$ is compact and $\ol{U}\subset V$. 
	\item For $f:X\to\Mb{R}$, $\supp(f):=\ol{\{f\ne0\}}$. We write $C_c(X):=\{f\in C(X)\mid \supp(f)\text{ is compact}\}$.
	\item We define \begin{gather*}
		\LIP(f)=\sup_{x,y\in X, x\ne y}\frac{|f(x)-f(y)|}{d(x,y)}\quad (f:X\to\Mb{R}),\\
		\LIP(X)=\LIP(X,d):=\{f:X\to\Mb{R}\mid\LIP(f)<\infty\}.
	\end{gather*}
	In particular, if $X$ is a subset of a Euclidean space, we use the Euclidean distance unless otherwise stated.
\end{itemize}
\section{Main result}\label{secpre}
Throughout this paper, we work under the following assumption.
\begin{asm}\label{asmspace}
Assume that $(X,d)$ is a locally compact, separable metric space and $\Mf{m}$ is a Radon measure with full support on $(X,d)$. We fix any $p\in(1,\infty)$.
\end{asm}
To state our assumptions for the main theorem, we first define $p$-energy forms in a broader sense according to \cite{KSdiff}.
\begin{defi}[$p$-Energy form: {\cite[Definition 3.1]{KSdiff}}]
	Let $(\Omega,\Mc{M},\mu)$ be a measure space. Write $L^0(\Omega,\mu)$ or simply $L^0(\mu)$ for the space of real-valued $\Mc{M}$-measurable functions modulo $\mu$-a.e. equality. We call $(\Mc{E},\Mc{F})$ a \emph{$p$-energy form} on $(\Omega,\mu)$ if $\Mc{F}$ is a linear subspace of $L^0(\Omega,\mu)$ and $\Mc{E}:\Mc{F}\to[0,\infty)$ is a functional such that $\Mc{E}^{1/p}$ is a seminorm on $\Mc{F}$.
\end{defi}
We next introduce the $p$-Clarkson inequalities.
\begin{defi}[$p$-Clarkson inequalities]\label{defCLKpre}
	Let $\Phi$ be a nonnegative function on a vector space, and $A$ be a linear subspace of the domain of $\Phi$.
	We say that the \emph{$p$-Clarkson inequalities} for $\Phi$ hold on $A$ if
	\begin{itemize}
		\item $p\in(1,2)$ and \eqref{eqCLK1}--\eqref{eqCLK2} hold for all $u,v\in A$, or
		\item $p\in[2,\infty)$ and \eqref{eqCLK3}--\eqref{eqCLK4} hold for all $u,v\in A$,
	\end{itemize}
	where
	\begin{align}
		2\left(\Phi(u)^{\frac{p}{p-1}}+\Phi(v)^{\frac{p}{p-1}}\right)^{p-1}
		&\le \Phi(u+v)^p+\Phi(u-v)^p \label{eqCLK1}\tag{CI1}\\
		2\left(\Phi(u)^{p}+\Phi(v)^{p}\right)
		&\ge \Phi(u+v)^p+\Phi(u-v)^p \label{eqCLK2}\tag{CI2}
		\shortintertext{and}
		2\left(\Phi(u)^{\frac{p}{p-1}}+\Phi(v)^{\frac{p}{p-1}}\right)^{p-1}
		&\ge \Phi(u+v)^p+\Phi(u-v)^p \label{eqCLK3}\tag{CI3}\\
		2\left(\Phi(u)^{p}+\Phi(v)^{p}\right)
		&\le \Phi(u+v)^p+\Phi(u-v)^p. \label{eqCLK4}\tag{CI4}
	\end{align}
\end{defi}
These inequalities serve as alternatives to the parallelogram law for $L^p$-like seminorms. The $p$-Clarkson inequalities for the $L^p$ norm were proved in \cite{Clk}.\footnote{To be precise, only the case $L^p([0,1])$ was considered in \cite{Clk}, but the same proof works for general measure spaces.} Roughly speaking, \eqref{eqCLK1} and \eqref{eqCLK4} express uniform convexity, while \eqref{eqCLK2} and \eqref{eqCLK3} express uniform smoothness. In fact, the uniform convexity of $(\Mc{F},(\Mc{E}+\|\cdot\|^p_{L^p(\Mf{m})})^{1/p})$ and the differentiability of the map $t\mapsto\Mc{E}(f+tg)$ for $f,g\in\Mc{F}$ were proved in \cite{KSdiff}.

Our main theorem is the following.
\begin{thm}\label{thmunique}
Let $(\Mc{E},\Mc{F})$ be a $p$-energy form satisfying the following conditions.
	\begin{description}
	\Econd{defBanach} $\Mc{F}\subset L^p(\Mf{m})$. Moreover, $(\Mc{F},\|\cdot\|_{\Mc{E}_1})$ is a Banach space, where $\|\cdot\|_{\Mc{E}_1}=\bigl(\Mc{E}+\|\cdot\|^p_{L^p(\Mf{m})}\bigr)^{1/p}$.
	\Econd{defMarkov} Let $f\in\core$. Then $g:=(f\wedge 1)^+\in\core$ and $\Mc{E}(g)\le\Mc{E}(f)$. 
	\Econd{defstrloc} (Strong local property) $\Mc{E}(f+g)=\Mc{E}(f)+\Mc{E}(g)$ if $f,g\in\core$ and $\supp(f)\cap\supp(g+a)=\emptyset$ for some $a\in\Mb{R}$.
	\Econd{defreg}	(Regularity) $\core$ is dense in both $(\Mc{F},\|\cdot\|_{\Mc{E}_1})$ and $(C_c(X),\|\cdot\|_{\infty})$.
	\Econd{defclk} The $p$-Clarkson inequalities for $\Mc{E}^{1/p}$ hold on $\Mc{F}$. 
\end{description}
 Then there exists a unique family $\{\mu_{\langle f\rangle}\}_{f\in\Mc{F}}$ of Radon measures on $(X,d)$ satisfying \ref{condmeasene}--\ref{condmeasclk}; we call such a family \emph{canonical $p$-energy measures}.
	\begin{description}
		\Mcond{condmeasene} $\mu_{\langle f\rangle}(X)=\Mc{E}(f)$ for any $f\in\Mc{F}$.
		\Mcond{condmeasloc} If $f,g\in\core$, $A\in\Mc{B}(X)$ and $(f-g)|_A$ is constant, then $\mu_{\langle f\rangle}(A)=\mu_{\langle g\rangle}(A)$.
        \Mcond{condmeasnorm} For any $A\in\Mc{B}(X)$,\,the functional\,$f\mapsto\mu_{\langle f\rangle}(A)^{1/p}$ is a seminorm on $\Mc{F}$.
		\Mcond{condmeasclk} For any $A\in\Mc{B}(X)$, the $p$-Clarkson inequalities for the functional $f\mapsto\mu_{\langle f\rangle}(A)^{1/p}$ hold on $\Mc{F}$.
	\end{description}
	Moreover, this family satisfies \ref{condmeaschain} and \ref{condmeasLeibniz}.
	\begin{description}
		\Mcond{condmeaschain} (Chain rule.) For any $f, g\in\core$ and piecewise $C^1$ functions $\varphi,\psi:\Mb{R}\to\Mb{R}$ with $\varphi(0)=\psi(0)=0,$
		\begin{equation*}
			\der\mu_{\langle \varphi\circ f;\psi\circ g\rangle}
			=\sgn(\varphi'\circ f)\,|\varphi'\circ f|^{p-1}(\psi'\circ g)\,\der \mu_{\langle f;g\rangle}
		\end{equation*}
		holds, where $\mu_{\langle f;g\rangle}$ is a signed measure defined by
		$\mu_{\langle f;g\rangle}\!:=\frac{1}{p}\frac{\der}{\der t}\mu_{\langle f+tg\rangle}\big|_{t=0}$.
		\Mcond{condmeasLeibniz} (Leibniz rule.)
		$\der\mu_{\langle f;gh\rangle}=g\der\mu_{\langle f;h\rangle}+h\der\mu_{\langle f;g\rangle}$ holds for any $f,g,h\in\core$.
	\end{description}
\end{thm}
\begin{rem}\label{remthmunique}
\begin{enumerate}
	\item The assumptions of this theorem are $p$-analogs of the properties of strongly local, regular Dirichlet forms. Condition \ref{defMarkov} corresponds to a form of the Markov property of Dirichlet forms (see, e.g., \cite{FOT}). For other equivalent formulations of \ref{defMarkov}, see Proposition \ref{propnMarkov} and Remark \ref{remWMarkov}. 
	\item This theorem is applicable to many fractal settings, and recovers the results in \cite{AES,AB,CheGYZ,KSdiff,KSKS,KS25,KSx,MS,Shi}: see Example \ref{exSC} below. Other applicable cases are discussed in Section \ref{secKS}.
	\item Contrary to Example \ref{exSC}, the upper-gradient method is generally not useful for analysis on fractals.
	For example, the Newtonian space $\Mc{N}^{1,p}(X,\Mf{m})$ on the standard Sierpi\'nski carpet coincides with $L^p(X,\Mf{m})$; see \cite[Proposition 4.3.3]{MT} and \cite[Proposition 7.1.33]{HKST}.
	\item Partial results still hold for cases where \ref{defclk} fails (cf.\ \cite{Eri, EM}): see Theorem \ref{thmdomination}.
\end{enumerate}
\end{rem}
\begin{ex}\label{exSC}
Let $\Phi:2^{\Mb{R}^2}\to2^{\Mb{R}^2}$ be defined by $\Phi(A)=\Cup_{s\in S}\varphi_s(A)$, where $S=\{-1,0,1\}^2\setminus\{(0,0)\}$ and $\varphi_s$ is the $\frac{1}{3}$-scaling centered at $s$. The unique nonempty compact set $X$ with $\Phi(X)=X$ is called the standard Sierpi\'nski carpet (Figure \ref{figSC}). We write $\varphi_{s_1s_2\dots s_n}=\varphi_{s_1}\circ\varphi_{s_2}\circ\cdots\circ\varphi_{s_n}$ for any $n\ge1$ and $s_1s_2\dots s_n\in S^n$. Set $d$ to be (the restriction of) the Euclidean distance and $\Mf{m}$ to be the unique Radon measure on $(X,d)$ with $\Mf{m}(\varphi_{s_1s_2\dots s_n}(X))=8^{-n}$ for any $n\ge1$ and $s_1s_2\dots s_n\in S^n$. For $n\ge1$, we let 
\[E_n=\left\{(u,v)\in S^n\times S^n\mid u\ne v\text{ and }\varphi_u(X)\cap \varphi_v(X)\ne\emptyset\right\}.\]
Define $M_n:L^0(\Mf{m})\to\Mb{R}^{S^n}$ and $\Mc{E}_n:L^0(\Mf{m})\to[0,\infty)$ by
\[(M_nf)(u)=\frac{\int_{\varphi_u(X)}f\der \Mf{m}}{\Mf{m}(\varphi_u(X))},\quad \Mc{E}_n(f)=\frac{1}{2}\sum_{(u,v)\in E_n}\left|(M_nf)(u)-(M_nf)(v)\right|^p.\]
Then, there exist $C,\rho_p>0$ and a $p$-energy form $(\Mc{E},\Mc{F})$ such that
\begin{gather*}
\Mc{F}=\left\{f\in L^p(\Mf{m})\mid \sup_{n}\rho_p^n\Mc{E}_n(f)<\infty \right\},\\ C^{-1}\sup_n\rho_p^n\Mc{E}_n(f)\le \Mc{E}(f)\le C\liminf_{n\to\infty}\rho_p^n\Mc{E}_n(f)
\end{gather*}
for any $f\in\Mc{F}$, and \ref{defBanach}--\ref{defclk} hold. Moreover, $(\Mc{E},\Mc{F})$ has self-similarity and symmetry in the following sense.
\begin{itemize}
    \item For any $f\in\Mc{F}$ and $s\in S$, we obtain $f\circ\varphi_s\in\Mc{F}$ and $\Mc{E}(f)=\rho_p\sum_{s\in S}\Mc{E}(f\circ\varphi_s)$.
    \item For any $f\in\Mc{F}$ and isometry $\Xi:X\to X$, we have $f\circ\Xi\in\Mc{F}$ and $\Mc{E}(f)=\Mc{E}(f\circ\Xi)$.
\end{itemize}
See \cite{MS} for further details and generalization.
\end{ex}

\section{Domination principle}\label{secunique}
In this section, we prove the domination principle for $p$-energy measures, which is the key ingredient for the proof of uniqueness statement in Theorem \ref{thmunique}.
\begin{thm}[Domination principle]\label{thmdomination}
	Let $(\Mc{E},\Mc{F}), (\tilde{\Mc{E}},\Mc{F})$ be $p$-energy forms with a common domain, and let $\{\mu_{\langle f\rangle}\}_{f\in\Mc{F}},\{\nu_{\langle f\rangle}\}_{f\in\Mc{F}}$ be families of Radon measures. Assume that both $(\Mc{E},\Mc{F})$ and $(\tilde{\Mc{E}},\Mc{F})$ satisfy \ref{defBanach}--\ref{defreg}, and that $\{\mu_{\langle f\rangle}\}_{f\in\Mc{F}}$ (resp. $\{\nu_{\langle f\rangle}\}_{f\in\Mc{F}}$) satisfies \ref{condmeasene}, \ref{condmeasloc} and the following conditions \ref{condmeaspm} and \ref{condmeaslim} with respect to $(\Mc{E},\Mc{F})$ (resp. $(\tilde{\Mc{E}},\Mc{F})$).
	\begin{description}
	\Mcond{condmeaspm} $\mu_{\langle f\rangle}=\mu_{\langle -f\rangle}$ for any $f\in\Mc{F}$.
\Mcond{condmeaslim} For $f,f_n\in\Mc{F}\ (n\ge1)$, if $\lim_{n\to\infty}\Mc{E}(f-f_n)=0$, then $\lim_{n\to\infty}\mu_{\langle f_n\rangle}(A)=\mu_{\langle f\rangle}(A)$ for any $A\in\Mc{B}(X)$.
	\end{description}
	 Then, $\mu_{\langle f\rangle}(A)\le \nu_{\langle f\rangle}(A)$ for any $f\in\Mc{F}$ and $A\in\Mc{B}(X)$, whenever $\Mc{E}(f)\le\tilde{\Mc{E}}(f)$ for any $f\in\core$.
\end{thm}
The proof of Theorem \ref{thmdomination} is essentially divided into two steps: the first step is to prove the limit formula \eqref{eqlimmeas}, while the second step is a standard regularity argument. We define $T_n, S_n^a \in C(\Mb{R})\ (n\in\Mb{N}, a\in\Mb{R})$ by
\begin{gather*}
	T_n(t):=\min_{k\in\Mb{Z}}\left|t-2^{-(n-1)}k\right|,\ \quad
	S_n^a(t):=\left((-t+ a+2^{-n})\wedge 2^{-n}\right)^+.
\end{gather*}
For $f,g\in\core$ and $a\in\Mb{R}$, We also define
\[f_n^{g,a}:=(T_n\circ f)\wedge (S_n^{a}\circ g).\]
For an open interval $I=(a,b)$, $\eta_I\in C(\Mb{R})$ denotes the shifted truncation map defined by
\[
\eta_{I}(t):=(t\vee a)\wedge b-(0\vee a)\wedge b.
\]

\begin{lem}\label{lemtrunc}
    Let $f\in\core$. Then $\eta_{(a,b)}\circ f\in \core$ for any open interval $(a,b)$. Moreover, for any $g\in\core$, $n\in\Mb{N}$, and $a\in\Mb{R}$, we have $T_n\circ f\in\core$ and $f_n^{g,a}\in\core$.
\end{lem}

\begin{proof}
	Clearly, $\eta_{(0,b)}\circ f, \eta_{(a,0)}\circ f\in\core$ for $a\le0\le b$ because $\|f\|_{\infty}<\infty$. Since
	\[
	\eta_{(a,b)}=\begin{cases}
		\eta_{(0,b)}-\eta_{(0,a)}& (a>0),\\
		\eta_{(0,b)}+\eta_{(a,0)} & (a\le0\le b),\\
		\eta_{(a,0)}-\eta_{(b,0)} &(b<0),
	\end{cases}
	\]
	we have $\eta_I\circ f\in \core$ for any open interval $I$.
	Let $M\in\Mb{N}$ with $M>\|f\|_{\infty}$. Then
	\[
	T_n\circ f=\sum_{k=-2^nM}^{2^nM-1}(-1)^k \eta_{I_{n,k}}\circ f\in\core,
	\quad \left(I_{n,k}:=(2^{-n}k,2^{-n}(k+1))\right).
	\]
	Finally, taking $I=(-a-2^{-n},-a)$ we obtain
	\begin{align*}
		f_n^{g,a}
		=&\left(T_n\circ f\right)\wedge\bigl(-\eta_I \circ(-g)+S_n^a(0)\bigr)\\
		=&\frac{T_n\circ f+(-\eta_I\circ(-g))}{2}
		-\frac{|T_n\circ f+(-\eta_I\circ(-g))-S_n^a(0)|-S_n^a(0)}{2} \\
		=&\begin{multlined}[t][.85\linewidth]
			\frac{1}{2}\bigl(T_n\circ f-\eta_I\circ(-g)
			+\eta_{(-\infty,-S_n^a(0))}\circ \left(T_n\circ f+\eta_I\circ(-g)\right)\\
			-\eta_{(-S_n^a(0),\infty)}\circ \left(T_n\circ f+\eta_I\circ(-g)\right)\bigr),
		\end{multlined}
	\end{align*}
	because $S_n^a(0)\ge0$. Hence $f_n^{g,a}\in \core$.
\end{proof}

\begin{lem}\label{lemlimmeas}
	Let $\{\mu_{\langle f\rangle}\}_{f\in\Mc{F}}$ be a family of Radon measures satisfying \ref{condmeasene}, \ref{condmeasloc} and \ref{condmeaspm}. Then
	\begin{equation}\label{eqlimmeas}
		\mu_{\langle f\rangle}(\{g\le a\})=\lim_{n\to\infty}\Mc{E}(f_n^{g,a})
	\end{equation}
	for any $f,g\in\core$ and $a\in\Mb{R}$.
\end{lem}

\begin{proof}
	Note that $\mu_{\langle T_n\circ f\rangle}=\mu_{\langle f\rangle}$ by \ref{condmeaspm} and \ref{condmeasloc}. Since
	\begin{gather*}
		\{g\le a\}\subset A_n:=\{T_n\circ f\le S_n^a\circ g,\ S_n^a\circ g\ne0\}\subset \{g\le a+2^{-n}\},\\
		\{0<S_n^a\circ g< T_n\circ f\}\subset\{a<g<a+2^{-n}\},
	\end{gather*}
	and $\mu_{\langle 0\rangle}=0$ by \ref{condmeasene}, we have,
	\begin{align*}
		&\mu_{\langle  f\rangle}(\{g\le a\})\\
		\le& \mu_{\langle T_n\circ f\rangle}(A_n)
		+\mu_{\langle -g\rangle}(\{0<S_n^a\circ g<T_n\circ f\})
		+\mu_{\langle 0\rangle}(\{S_n^a\circ g=0\})\\
		=& \Mc{E}(f_n^{g,a})\\
		\le& \mu_{\langle  f\rangle}(\{g\le a+2^{-n}\})
		+\mu_{\langle -g\rangle}(\{a<g<a+2^{-n}\}).
	\end{align*}
	(The equality follows from \ref{condmeasene} and \ref{condmeasloc}.) Letting $n\to\infty$, we obtain \eqref{eqlimmeas}.
\end{proof}

\begin{lem}\label{lemapproxopen}
	Let $U\subset X$ be open. Then there exist $g_i\in\core\ (i\in\Mb{N})$ such that $\{\{g_i\le -1/2\}\}_{i\ge 1}$ is an increasing sequence and $\cup_{i\ge 1}\{g_i\le -1/2\}=U$.
\end{lem}

\begin{proof}
	Since $(X,d)$ is locally compact and separable, there exists a sequence of open subsets $\{U_i\}_{i\ge 1}$ with $U_i\Subset U_{i+1}\ (i\ge1)$ and $\cup_{i\ge 1}U_i=U$.
	Let
	\[
	h_i(x)=0\wedge\left(\frac{d(U_i,x)}{d(U_i,X\setminus U_{i+1})}-1\right)\in C_c(X).
	\]
	Then there exists $g_i\in\core$ with $\|g_i-h_i\|_{\infty}\le 1/3$ by \ref{defreg}. Hence $\ol{U_i}\subset \{g_i\le -1/2\}\subset U_{i+1}$, and the claim follows.
\end{proof}

\begin{proof}[Proof of Theorem \ref{thmdomination}]
	For $f\in\core$ and an open set $A\subset X$, the claim follows from Lemmas \ref{lemlimmeas} and \ref{lemapproxopen}. The other cases reduce to this one by \ref{defreg}, \ref{condmeaslim}, and the outer regularity of Radon measures.
\end{proof}

\section{\texorpdfstring{Proof of Theorem \ref{thmunique} (Part I): basic properties of $p$-energy forms}{Basic properties of p-energy forms}}\label{secbas}
To prepare for the proof of the existence part of Theorem \ref{thmunique}, we establish several basic properties of elements of $\Mc{F}$ and of convergence in $(\Mc{F},\|\cdot\|_{\Mc{E}_1})$, including the Markov property for normal contractions (Proposition \ref{propnMarkov}). From this section to the end of Section \ref{secchain}, we work under the following assumption.
\begin{asm}\label{asmform}
$(\Mc{E},\Mc{F})$ is a $p$-energy form satisfying \ref{defBanach}--\ref{defclk}.
\end{asm}
First, we note that $(\Mc{F},\|\cdot\|_{\Mc{E}_1})$ is reflexive by \cite[Proposition 3.13]{KSdiff}; this essentially follows from \ref{defclk}, and will be used repeatedly.

\begin{lem}\label{lemconvconv}
	Let $f, f_i\in L^p(\Mf{m})\ (i\ge1)$ with $f_i\to f$ in norm. Assume that $f_i\in\Mc{F}\ (i\ge 1)$ and that
	\begin{equation}\label{eqconv}
		\infty>\liminf_{i,j\to\infty}\Mc{E}\left(\frac{f_i+f_j}{2}\right)\ge \limsup_{i\to\infty}\Mc{E}(f_i).
	\end{equation}
	Then $f\in\Mc{F}$ and $f_i\to f$ in $(\Mc{F},\|\cdot\|_{\Mc{E}_1})$. In particular, \eqref{eqconv} holds if
	$\Mc{E}(f_{i_3})\le \Mc{E}\left(\frac{f_{i_1}+f_{i_2}}{2}\right)$ for any $i_1\le i_2\le i_3$.
\end{lem}

\begin{proof}
	From \eqref{eqconv}, we obtain $\lim_{i,j\to\infty}\Mc{E}(f_i-f_j)=0$ because, by \ref{defclk},
	\[
	\Mc{E}\left(\frac{f_i-f_j}{2}\right)\le
	\begin{cases}
		\left(\left(\frac{\Mc{E}(f_i)+\Mc{E}(f_j)}{2}\right)^{\frac{1}{p-1}}
		-\Mc{E}\left(\frac{f_i+f_j}{2}\right)^{\frac{1}{p-1}}\right)^{p-1} & (1<p<2),\\[0.3em]
		\displaystyle \frac{\Mc{E}(f_i)+\Mc{E}(f_j)}{2}-\Mc{E}\left(\frac{f_i+f_j}{2}\right) & (2\le p).
	\end{cases}
	\]
	Hence $\{f_i\}_{i\ge1}$ is a Cauchy sequence in $(\Mc{F},\|\cdot\|_{\Mc{E}_1})$, and therefore $f_i\to g$ in $(\Mc{F},\|\cdot\|_{\Mc{E}_1})$ for some $g\in\Mc{F}$. By uniqueness of the $L^p$-limit,  we have $g=f$.
	For the last assertion, the condition $\Mc{E}(f_{i_3})\le \Mc{E}\!\left(\frac{f_{i_1}+f_{i_2}}{2}\right)$ clearly yields \eqref{eqconv}.
\end{proof}

\begin{lem}\label{lemfold}
	Let $\varphi\in C(\Mb{R})$, $f\in\core$, and $a_0,\dots,a_N\in\Mb{R}$ satisfy
	$a_0<a_1<\cdots<a_N$ and $a_0\le f\le a_N$.
	If $\varphi(0)=0$ and $\varphi|_{[a_{i-1},a_i]}$ is affine for each $1\le i\le N$, then $\varphi\circ f\in\core$ and
	\begin{equation}\label{eqefold}
		\Mc{E}(\varphi\circ f)=\sum_{i=1}^{N} \LIP(\varphi|_{[a_{i-1},a_i]})^p\,\Mc{E}\bigl(\eta_{(a_{i-1},a_i)}\circ f\bigr).
	\end{equation}
\end{lem}

A useful corollary immediately follows from this lemma.
\begin{cor}\label{corfold}
	Let $\varphi_1,\varphi_2$ be piecewise affine functions with $\varphi_1(0)=\varphi_2(0)=0$ and $|\varphi_1'|\le|\varphi_2'|$ a.e.~(with respect to the Lebesgue measure). Then $\Mc{E}(\varphi_1\circ f)\le\Mc{E}(\varphi_2\circ f)$ for any $f\in\core$.
\end{cor}

\begin{proof}[Proof of Lemma \ref{lemfold}]
	We may assume $\varphi=\sum_{i=1}^N b_i\eta_{(a_{i-1},a_i)}\ (b_i\in\Mb{R},1\le i\le N)$ without loss of generality. 
	Then $\varphi\circ f\in\core$ by Lemma \ref{lemtrunc}.\par
	We prove \eqref{eqefold} by induction on $N$. The case $N=1$ is immediate. Assume $N>1$.
	We first consider the case $a_1\le 0$. For sufficiently large $n$, define
	\begin{align*}
		&\varphi_n:=b_1\,\eta_{(a_0,\,a_1-2^{-n})}+\sum_{i=2}^{N} b_i\,\eta_{(a_{i-1},a_i)}.
		\shortintertext{Then}
		\Mc{E}(\varphi_n\circ f)
		&=\Mc{E}\bigl(b_1\eta_{(a_0,a_1-2^{-n})}\circ f\bigr)
		+\Mc{E}\left(\sum_{i=2}^{N} b_i\eta_{(a_{i-1},a_i)}\circ f\right)\\
		&=|b_1|^p\,\Mc{E}\bigl(\eta_{(a_0,a_1-2^{-n})}\circ f\bigr)
		+\Mc{E}\bigl((\varphi\circ\eta_{(a_1,a_N)})\circ f\bigr)\\
		&=|b_1|^p\,\Mc{E}\bigl(\eta_{(a_0,a_1-2^{-n})}\circ f\bigr)
		+\sum_{i=2}^{N} |b_i|^p\,\Mc{E}\bigl(\eta_{(a_{i-1},a_i)}\circ f\bigr),
	\end{align*}
	by \ref{defstrloc} and the induction hypothesis.
	Moreover,
	\[
	\varphi\circ f-\varphi_n\circ f
	=b_1\bigl(\eta_{(a_0,a_1)}-\eta_{(a_0,a_1-2^{-n})}\bigr)\circ f
	=b_1\,\eta_{(a_1-2^{-n},a_1)}\circ f\to 0 \quad\text{in }L^p(\Mf{m})
	\]
	by the dominated convergence theorem. In addition, for any $n_1\le n_2\le n_3$, \ref{defMarkov} yields
	\[
	\Mc{E}\bigl(b_1\eta_{(a_1-2^{-n_3},a_1)}\circ f\bigr)
	\le
	\Mc{E}\left(\frac{b_1\eta_{(a_1-2^{-n_1},a_1)}\circ f+b_1\eta_{(a_1-2^{-n_2},a_1)}\circ f}{2}\right).
	\]
	Hence $\Mc{E}(\varphi\circ f)=\lim_{n\to\infty}\Mc{E}(\varphi_n\circ f)$ and
	$\lim_{n\to\infty}\Mc{E}\bigl(\eta_{(a_0,a_1-2^{-n})}\circ f\bigr)=\Mc{E}\bigl(\eta_{(a_0,a_1)}\circ f\bigr)$
	by Lemma \ref{lemconvconv}. Therefore \eqref{eqefold} holds in this case.\par
	For the case $a_1>0$, let $\phi(t):=\varphi(-t).$ Then $-a_{N-1}< 0,$ and 
	\begin{multline*}
		\Mc{E}(\varphi\circ f)=\Mc{E}(\phi\circ(-f))\\=\sum_{i=1}^N|-b_i|^p\Mc{E}(\eta_{(-a_i,-a_{i-1})}\circ (-f))=\sum_{i=1}^N|b_i|^p\Mc{E}(\eta_{(a_{i-1},a_i)}\circ f)
	\end{multline*}
	follows from the previous case. 
\end{proof}

\begin{prop}\label{propnMarkov}
	Let $\varphi\in C(\Mb{R})$ be a normal contraction; that is, $\varphi(0)=0$ and $\LIP(\varphi)\le1$. Then $\varphi\circ f\in \Mc{F}$ and $\Mc{E}(\varphi\circ f)\le\Mc{E}(f)$ for any $f\in\Mc{F}$.
\end{prop}

\begin{proof}
	Let $\varphi_i:\Mb{R}\to\Mb{R}$ be such that $\varphi_i|_{[k/2^i,(k+1)/2^i]}$ is affine for every $k\in\Mb{Z}$ and $\varphi_i=\varphi$ on $2^{-i}\Mb{Z}$. Then $\varphi_i\to \varphi$ pointwise on $\Mb{R}$ and $\LIP(\varphi_i)\le1$.
	Let $f_j\in\core$ satisfy $\|f_j-f\|_{\Mc{E}_1}\to 0$ as $i\to\infty$.
	Choose a subsequence $\{f_{j_k}\}_{k\ge 1}$ such that $\|f_{j_k}-f\|_{L^p(\Mf{m})}\le 2^{-k}$.
	Then
	$
	\left\|\limsup_{k\to\infty}|f_{j_k}-f|\right\|_{L^p(\Mf{m})}
	\le \lim_{k\to\infty}\|\sum_{l\ge k} |f_{j_l}-f|\|_{L^p(\Mf{m})}
	\le \lim_{k\to\infty}2^{-k+1}=0,
	$
	and hence $f_{j_k}\to f$ $\Mf{m}$-a.e. By the dominated convergence theorem, $\varphi_k\circ f_{j_k}\to \varphi\circ f$ in $L^p(\Mf{m})$, because $\varphi_k\circ f_{j_k}\to \varphi\circ f$ and
	$
	|\varphi_k\circ f_{j_k}|
	\le \LIP(\varphi_k)\,|f_{j_k}|
	\le |f|+\sum_{l\ge1}|f-f_{j_l}|
	$
	$\Mf{m}$-a.e. 
	Therefore, $\varphi\circ f\in\Mc{F}$ and
	\[
	\Mc{E}(\varphi\circ f)\le \liminf_{k\to\infty}\Mc{E}(\varphi_{k}\circ f_{j_k})
	\le \lim_{k\to\infty}\Mc{E}(f_{j_k})=\Mc{E}(f),
	\]
	by \cite[Proposition 3.18(a)]{KSdiff} and Lemma \ref{lemfold}.
\end{proof}

\begin{rem}\label{remWMarkov}
	Proposition \ref{propnMarkov} corresponds to one characterization of the Markov property of Dirichlet forms (see, e.g., \cite[Section 1.1]{FOT}). We also note that the following condition \ref{condWMarkov} corresponds to another characterization of the Markov property.
	\setcounter{energydashcond}{1} 
	\begin{description}
		\Edashcond{condWMarkov} For any $f\in\core$ and $\epsilon>0$, there exists a non-decreasing function $\varphi:\Mb{R}\to\Mb{R}$ such that $\varphi(t)=t$ on $[0,1]$, $-\epsilon\le \varphi\le 1+\epsilon$, $|\varphi(a)-\varphi(b)|\le|a-b|$ for any $a,b\in\Mb{R}$, $\varphi\circ f\in\core$, and $\Mc{E}(\varphi\circ f)\le\Mc{E}(f)$.
	\end{description}
	Under Assumptions \ref{asmspace} and \ref{asmform} except for \ref{defMarkov}, \ref{condWMarkov} is equivalent to \ref{defMarkov} by \cite[Proposition 3.18(a)]{KSdiff}.
\end{rem}

\begin{cor}\label{corincl}
	\begin{enumerate}
		\item If $f\in\Mc{F}$, then $\eta_I\circ f$, $|f-a|-|a|\in\Mc{F}$ and $\Mc{E}(\eta_I\circ f)\le\Mc{E}(f)$, $\Mc{E}(|f-a|-|a|)\le\Mc{E}(f)$ for any open interval $I$ and $a\in\Mb{R}$.
		\item If $f,g\in\Mc{F}$ {\rm(}resp.\ $\core${\rm)} and $a\ge0$, then $f\vee (g-a),\ f\wedge (g+a)\in\Mc{F}$ {\rm(}resp.\ $\core${\rm)} and
		\[
		\Mc{E}(f\vee (g-a))\vee\Mc{E}(f\wedge (g+a))
		\le\left(\Mc{E}(f)^{1/p}+\Mc{E}(g)^{1/p}\right)^p.
		\]
		\item If $f\in\core$ and $\varphi$ is locally Lipschitz with $\varphi(0)=0$, then $\varphi\circ f\in \core$ and $\Mc{E}(\varphi\circ f)\le\LIP(\varphi|_{[-\|f\|_{\infty},\|f\|_{\infty}]})^p\Mc{E}(f)$.
		\item If $f,g\in\core$, then $fg\in\core$.
	\end{enumerate}
\end{cor}

\begin{proof}
	\begin{enumerate}
		\item This is immediate from Proposition \ref{propnMarkov}.
		\item We have
		\[\textstyle 
		f\vee (g-a)=\frac{f+(g-a)}{2}+\frac{|f-(g-a)|}{2}
		=\frac{f+g}{2}+\left(\frac{|(f-g)+a|}{2}-\frac{a}{2}\right)\in\Mc{F},
		\]
		and
		\begin{align*}
			\textstyle
			\Mc{E}\left(\frac{f+g}{2}+\left(\frac{|(f-g)+ a|}{2}-\frac{a}{2}\right)\right)^{\frac{1}{p}}\le& \textstyle \Mc{E}\left(\frac{f+g}{2}\right)^{\frac{1}{p}}+\Mc{E}\left(\frac{|(f-g)+a|}{2}-\frac{a}{2} \right)^{\frac{1}{p}}\\ \le&\textstyle\Mc{E}\left(\frac{f+g}{2} \right)^{\frac{1}{p}}+\Mc{E}\left(\frac{f-g}{2} \right)^{\frac{1}{p}}\\  
			\le&\Mc{E}(f)^{\frac{1}{p}}+\Mc{E}(g)^{\frac{1}{p}},
		\end{align*}
		by \ref{defBanach} and part (1). The claim for $f\wedge (g+a)=-((-f)\vee (-g-a))$ follows similarly.
		\item Since $\varphi\circ f=(\varphi\circ \eta_{(-\|f\|_{\infty},\,\|f\|_{\infty})})\circ f$, the claim follows from Proposition \ref{propnMarkov}.
		\item $fg=\frac{(f+g)^2-f^2-g^2}{2}\in\core$ by (3).
	\end{enumerate}
\end{proof}

\begin{lem}\label{lemV}
	Let $\Mc{V}$ be the linear span of $(\core)\cup\{\indi_{X}\}$ in $C(X)$. If $f,g\in \Mc{V}$ and $\varphi\in C(\Mb{R})$ is locally Lipschitz, then $\varphi\circ f,f\vee g, f\wedge g$ and $fg$ belong to $\Mc{V}$. Furthermore, $\Mc{V}\cap C_c(X)=\core$.
\end{lem}

\begin{proof}
	Throughout this proof, we write $f=u+a$ with $u\in\core$ and $a\in\Mb{R}$.
	For $\varphi$, let $\psi(t)=\varphi(t+a)-\varphi(a)$. Then $\psi$ is locally Lipschitz and $\psi(0)=0$.
	Hence $\psi\circ u\in\core$ by Corollary \ref{corincl}~(3), and thus
	$\varphi\circ f=\psi\circ u+\varphi(a)\in \Mc{V}$.\par
	Similarly to the proof of Corollary \ref{corincl}~(2), the inclusions $f\vee g, f\wedge g\in \Mc{V}$ follow from $|f-g|\in \Mc{V}$. The inclusion $fg\in \Mc{V}$ follows from Corollary \ref{corincl}~(4).\par
	We next prove that $\Mc{V}\cap C_c(X)=\core$.
	If $X$ is compact, then $\|2\indi_X-v\|_{\infty}<1$ for some $v\in\core$ by \ref{defreg}.
	Hence $\indi_X=\eta_{(-\infty,1)}v\in\core$, and therefore $\core=\Mc{V}=\Mc{V}\cap C_c(X)$.
	If $X$ is not compact, $u+a\in C_c(X)$ forces $a=0$ because $\{u=0\}$ is not compact. Thus $\Mc{V}\cap C_c(X)\subset \core$.
\end{proof}

\begin{cor}\label{corV}
	Let $f\in \Mc{V}$. If there exist $n\ge 1$, a function $F:\Mb{R}^n\to\Mb{R}$, and $\{f_i\}_{i=1}^n\in(\core)^n$ such that $f(x)=F(f_1(x),\dots,f_n(x))$, then $f-F(\zero)\in\core$.
\end{cor}

\begin{proof}
	Since $f|_{\cap_{i=1}^n \{f_i=0\}}\equiv F(\zero)$, we have
	$\{f\ne F(\zero)\}\subset\cup_{i=1}^n\{f_i\ne 0\}$.
	Hence $\supp(f-F(\zero))\subset \cup_{i=1}^n \supp(f_i)$ is compact, and thus $f-F(\zero)\in\core$ by Lemma \ref{lemV}.
\end{proof}

Hereafter, for $f\in\Mc{V}$ we define $\Mc{E}^\Mc{V}(f):=\Mc{E}(g)$, where $g\in\core$ and $f-g$ is constant: this is well-defined because if $\indi_{X}\in\core$ then $\Mc{E}(\indi_{X})=0$ by \ref{defstrloc} for $f=g=\indi_{X}$. \footnote{One would like to use $\Mc{E}$ instead of $\Mc{E}^\Mc{V}$ if we knew $\Mc{E}(1)=0$ whenever $\indi_X\in\Mc{F}$. However, we do not even know whether \ref{defstrloc} is a sufficient form of strong locality to prove it or not.}
We write
$f_n^{g+b,a}:=(T_n\circ f)\wedge (S_n^a\circ (g+b))=f_n^{g,a-b}\in\core
$
for $f,g\in\core$, $n\in\Mb{N}$, and $a,b\in\Mb{R}$.

\section{\texorpdfstring{Proof of Theorem \ref{thmunique} (Part II): construction of $p$-energy measures for $f\in\core$}{Construction of p-energy measures for core functions}}\label{secconst}
In this section, for each $f\in\core$ we construct $\mu_{\langle f \rangle}$ with $\mu_{\langle f\rangle}(X)=\Mc E(f)$. Note that we still assume Assumption \ref{asmform}.

\begin{lem}\label{lemfnaconv}
	For any $f\in\core$, $g\in\Mc{V}$, and $a\in\Mb{R}$, we have
	\[
	\lim_{n\to\infty}\Mc{E}^\Mc{V}(S_n^a\circ g)=0
	\quad\text{and}\quad
	\lim_{n\to\infty}\Mc{E}(f_n^{g,a})=\inf_{n\ge 1}\Mc{E}(f_n^{g,a}).
	\]
\end{lem}

\begin{proof}
	For $n,m\in\Mb{N}$ with $m>n$, set
	\[
	f_{n,m}(x)=\left(\left(\min_{k\in\Mb{Z}} \bigl|f(x)-(2k+1)2^{-n}\bigr|\right)\wedge 2^{-m}\right)-2^{-m}.
	\]
	Then $f_{n,m}\in\core$, $\lim_{m\to\infty}f_{n,m}(x)=0$, and $|f_{n,m}(x)|\le |f(x)|$ for any $n,m$ and $x\in X$.
	Moreover, by Corollary \ref{corfold},
	\[
	\Mc{E}(f_{n,m_3})\le\Mc{E}\left(\frac{f_{n,m_1}+f_{n,m_2}}{2}\right)
	\quad\text{for any } m_3\ge m_2\ge m_1>n.
	\]
	Hence, for each fixed $n$, Lemma \ref{lemconvconv} (together with the dominated convergence theorem) yields
	$\lim_{m\to\infty}\Mc{E}(f_{n,m})=0$. We obtain 
	\[
	\lim_{n\to\infty}\Mc{E}^\Mc{V}\left(\deco{S_n^a\circ g}\right)=\lim_{n\to\infty}\Mc{E}(-\eta_{(-a-2^{-n},-a)}\circ(-g))=0
	\] in the same way.\par
	Since
	$
	f_m^{g,a}=(T_m\circ f_n^{g,a})\wedge (S_m^a\circ g)\wedge(f_{n,m}+2^{-m})$,
	we have
	\begin{align*}
		0\le\Mc{E}(f_m^{g,a})
		&\le \Bigl(\Mc{E}(T_m\circ f_n^{g,a})^{\frac{1}{p}}
		+\Mc{E}^\Mc{V}\bigl(\deco{S_m^a\circ g}\bigr)^{\frac{1}{p}}
		+\Mc{E}(f_{n,m})^{\frac{1}{p}}\Bigr)^p\\
		&= \Bigl(\Mc{E}(f_n^{g,a})^{\frac{1}{p}}
		+\Mc{E}^\Mc{V}\bigl(\deco{S_m^a\circ g}\bigr)^{\frac{1}{p}}
		+\Mc{E}(f_{n,m})^{\frac{1}{p}}\Bigr)^p,
	\end{align*}
	by Lemma \ref{lemfold} and Corollary \ref{corincl}~(2).
	Taking $\limsup_{m\to\infty}$ and using the above limits, we obtain
	$
	\limsup_{m\to\infty}\Mc{E}(f_{m}^{g,a})\le \Mc{E}(f_n^{g,a})
	$, which implies
	$\lim_{m\to\infty}\Mc{E}(f_{m}^{g,a})=\inf_{n\ge1}\Mc{E}(f_n^{g,a})$.
\end{proof}

\begin{lem}\label{lemfnadist}
	Let $F^g_f(a)$ denote $\lim_{n\to\infty}\Mc{E}(f_n^{g,a})$ for $f\in\core$, $g\in\Mc{V}$, and $a\in\Mb{R}$. Then $F^g_f(\cdot)$ is the distribution function of a finite measure, and $\lim_{a\to\infty}F^g_f(a)=\Mc{E}(f)$.
\end{lem}

\begin{proof}
	\begin{itemize}
		\item (Monotonicity.) Let $a<b$. Then $f_n^{g,a}=f_n^{g,b}\wedge (S_n^a\circ g)$ for any $n$ with $2^{-n}< b-a$; hence,
		\[
		F^{g}_f(a)\le \lim_{n\to\infty}\Bigl(\Mc{E}(f_n^{g,b})^{1/p}+\Mc{E}^\Mc{V}\bigl(\deco{S_n^a\circ g}\bigr)^{1/p}\Bigr)^p
		=F^{g}_f(b),
		\]
		by Lemma \ref{lemfnaconv}.
		\item (Right continuity.) First, we prove that for any fixed $n\in\Mb{N}$ and $a\in\Mb{R}$,
		\[
		f_{n}^{g, a+2^{-m}}\longrightarrow f_n^{g,a}\quad\text{in the }\Mc{E}_1\text{-norm as }m\to\infty.
		\]
		Set $\ell_m:=f_{n}^{g,a+2^{-m}}-f_n^{g,a}$. Then $|\ell_m|\le 2^{-m}$ and also
		$|\ell_m|\le |f_{n}^{g, a+2^{-m}}|+|f_n^{g,a}|\le 2\,T_n\circ f$.
		Moreover, for any $n<m_1\le m_2\le m_3$,
		\[
		\ell_{m_3}
		=\left(\left(\frac{\ell_{m_1}+\ell_{m_2}}{2}\right)\wedge 2^{-m_3}\right)\wedge \left(S_{m_3}^{a+2^{-n}}\circ g\right),
		\]
		and hence, by Corollary \ref{corincl}~(2),
		\[
		\Mc{E}(\ell_{m_3})\le\Biggl(\Mc{E}\left(\frac{\ell_{m_1}+\ell_{m_2}}{2}\right)^{\frac{1}{p}}
		+\Mc{E}^\Mc{V}\left(\deco{S_{m_3}^{a+2^{-n}}\circ g}\right)^{\frac{1}{p}}\Biggr)^p.
		\]
		Using Lemma \ref{lemconvconv} together with the dominated convergence theorem, we conclude that
		$f_{n}^{g, a+2^{-m}}\to f_n^{g,a}$ in the $\Mc{E}_1$-norm as $m\to\infty$.
		By Lemma \ref{lemfnaconv}, it follows that
		\[
		F^g_f(a)=\lim_{n\to\infty}\lim_{m\to\infty}\Mc{E}(f_{n}^{g, a+2^{-m}})
		\ge \limsup_{m\to\infty} F^g_f(a+2^{-m}).
		\]
		Together with monotonicity, this yields the right continuity of $F^g_f$.
		\item Since $g$ is bounded, we have $S_{n}^{\|g\|_\infty}\circ g=2^{-n}$ and $S_{n}^{-\|g\|_\infty-1}\circ g=0$ on $X$.
		Moreover, $\Mc{E}(T_n\circ f)=\Mc{E}(f)$ by Lemma \ref{lemfold}.
		Therefore,
		$
		\lim_{a\to\infty}F^g_f(a)=\Mc{E}(f)
		$ and $
		\lim_{a\to-\infty}F^g_f(a)=0
		$
		by definition.
	\end{itemize}
\end{proof}

\begin{lem}\label{lemmeascap}
	Let $a<b<c$, $f\in\core$, and $g\in\Mc{V}$. Then
	\[
	\lim_{n\to\infty}\Mc{E}\bigl((T_n\circ f)\wedge(S_n^c\circ g)\wedge (S_n^{-b}\circ (-g))\bigr)\le F^g_f(c)-F^g_f(a).
	\]
\end{lem}

\begin{proof}
	Since $(T_n\circ f)\wedge (S_n^c\circ g)\wedge (S_n^{-b}\circ (-g))=f_n^{h,(c-b)/2}$ for $h=\deco{\bigl|g-\frac{b+c}{2}\bigr|}$, the limit exists.
	If $2^{-(n-1)}<b-a$, then
	\[
	f_n^{g,a}+\left((T_n\circ f)\wedge (S_n^c\circ g)\wedge (S_n^{-b}\circ (-g))\right)
	=f_n^{g,c}\wedge \left((S_n^a\circ g)+ (S_n^{-b}\circ (-g))\right).
	\]
	From this, together with \ref{defstrloc} and Corollary \ref{corincl}~(2), we obtain
	\begin{align*}
		&\Mc{E}\left((T_n\circ f)\wedge (S_n^c\circ g)\wedge (S_n^{-b}\circ (-g))\right)\\
		=&\Mc{E}\left(f_n^{g,a}+\left((T_n\circ f)\wedge (S_n^c\circ g)\wedge (S_n^{-b}\circ (-g))\right)\right)-\Mc{E}(f_n^{g,a})\\
		\le&\left(\Mc{E}(f_n^{g,c})^{\frac{1}{p}}
		+\Mc{E}^\Mc{V}\left(\deco{S_n^a\circ g}\right)^{\frac{1}{p}}
		+\Mc{E}^\Mc{V}\left(\deco{S_n^{-b}\circ(-g)}\right)^{\frac{1}{p}} \right)^p-\Mc{E}(f_n^{g,a}).
	\end{align*}
	Letting $n\to\infty$ yields the desired inequality.
\end{proof}

\begin{lem}\label{lemFcover}
	Let $N\in\Mb{N}$, $f\in\core$, and $g, h_i \in \Mc{V}\ (1\le i\le N)$ such that
	$\{g\le a\}\subset \cup_{i=1}^N \{h_i\le  b_i\}$ for some $a, b_i \in\Mb{R}\ (1\le i\le N)$.
	Then $F^g_f(a)\le\sum_{i=1}^N F_f^{h_i}(b_i)$.
	In particular, $F^g_f(a)\le F^{h_1}_f(b_1)$ if $\{g\le a\}\subset\{h_1\le b_1\}$.
\end{lem}

\begin{proof}
	We may assume $g, h_i\in\core$ without loss of generality.\par
	We first consider the case $N=2$. We prove that for any $b'_i>b_i\ (i=1,2)$, there exists $a'>a$ such that
	\[
	\{g\le a'\}\subset \{h_1\le b'_1\}\cup\{h_2\le b'_2\}.
	\]
	If $a<0$, then $\bigl\{\{g\le a'\}\setminus(\{h_1< b'_1\}\cup \{h_2< b'_2\})\bigr\}_{a'\in(a,0)}$
	is a decreasing family of compact sets whose intersection is empty. Hence, there exists $a'\in(a,0)$ with
	$\{g\le a'\}\subset \{h_1< b'_1\}\cup \{h_2< b'_2\}$.
	Similarly, if $a\ge0$, then there exists $a'>a$ such that
	$
	(\{g\le a'\}\cap \supp(g))\subset \{h_1< b'_1\}\cup \{h_2< b'_2\}.
	$
	Since $X\setminus\supp(g)\subset \{g\le a\}\subset \{h_1< b'_1\}\cup \{h_2< b'_2\}$, it follows that
	$\{g\le a'\}\subset \{h_1< b'_1\}\cup \{h_2< b'_2\}$.\par
	
	Fix $b_1'>b_1$, $b_2'>b_2$, and choose $a'>a$ as above.
	Let $n\in\Mb{N}$ and $\epsilon>0$ satisfy $2^{-n+1}<(a'-a)\wedge \epsilon$, and set
	\[
	u_\epsilon:=(-h_1+b'_1+2\epsilon)\vee(h_2-b'_2).
	\]
	Then $\supp\bigl(f_n^{h_1,b_1'+\epsilon}\bigr)\cap \supp\bigl(f_n^{u_\epsilon,0}\bigr)=\emptyset$
	because $\{u_\epsilon\le 2^{-n}\}\subset\{h_1>b_1'+\frac{3}{2}\epsilon\}$.
	Hence,
	\begin{align*}
		\Mc{E}(f_n^{h_1,b'_1+\epsilon}+f_n^{\deco{u_\epsilon},0})
		=&\Mc{E}(f_n^{h_1,b'_1+\epsilon})+\Mc{E}(f_n^{\deco{u_\epsilon},0})\\
		\le&\Mc{E} (f_n^{h_1,b'_1+\epsilon})
		+\left(\Mc{E}(f_n^{h_2,b'_2+\epsilon})^{\frac{1}{p}}+\Mc{E}^\Mc{V}(S_n^0\circ \deco{u_\epsilon})^{\frac{1}{p}} \right)^p
	\end{align*}
	by \ref{defstrloc} and Corollary \ref{corincl}~(2), because $f_n^{h_2,b_2'+\epsilon}=T_n\circ f$ on $\{u_\epsilon\le 2^{-n}\}$.
	
	Moreover, let
	\[
	v_\epsilon:=(T_n\circ f)\wedge (S_n^{b'_1+2\epsilon}\circ h_1)\wedge (S_n^{-(b'_1+\epsilon)}\circ(-h_1)).
	\]
	Then
	\begin{gather*}
		\begin{aligned}
			\left(\left(f_n^{h_1,b'_1+\epsilon}+f_n^{\deco{u_\epsilon},0}\right)\vee v_\epsilon \right)\Big\vert_{\{h_1\le b'_1+2\epsilon\}}
			&=\left(f_n^{h_1,b_1'+\epsilon}\vee v_\epsilon\right)\Big\vert_{\{h_1\le b_1'+2\epsilon\}}\\
			&=(T_n\circ f)\vert_{\{h_1\le b_1'+2\epsilon\}},
		\end{aligned}
		\shortintertext{and}
		\begin{aligned}
			\left(\left(f_n^{h_1,b'_1+\epsilon}+f_n^{\deco{u_\epsilon},0}\right)\vee v_\epsilon \right)\Big\vert_{\{h_2\le b'_2\}\setminus\{h_1\le b'_1+2\epsilon\}}
			&=f_n^{u_\epsilon,0}\vert_{\{h_2\le b'_2\}\setminus\{h_1\le b'_1+2\epsilon\}}\\
			&=(T_n\circ f)\vert_{\{h_2\le b'_2\}\setminus\{h_1\le b'_1+2\epsilon\}}.
		\end{aligned}
	\end{gather*}
	Hence,
	\[
	f_n^{g,a}=\Bigl(\bigl(f_n^{h_1,b'_1+\epsilon}+f_n^{\deco{u_\epsilon},0}\bigr)\vee v_\epsilon\Bigr)\wedge (S_n^a\circ g).
	\]
	Therefore,
	\begin{align*}
		\lim_{n\to\infty}\Mc{E}(f_n^{g,a})
		\le&\lim_{n\to\infty}\left(\Mc{E}\left(f_n^{h_1,b'_1+\epsilon}+f_n^{\deco{u_\epsilon},0}\right)^{\frac{1}{p}}
		+\Mc{E}(v_\epsilon)^{\frac{1}{p}}+\Mc{E}^\Mc{V}\left(\deco{S_n^a\circ g}\right)^{\frac{1}{p}} \right)^p\\
		\le&\left(\left(F_f^{h_1}(b'_1+\epsilon)+F_f^{h_2}(b'_2+\epsilon) \right)^{\frac{1}{p}}
		+\left(F_f^{h_1}(b'_1+2\epsilon)-F_f^{h_1}(b'_1) \right)^{\frac{1}{p}}\right)^p,
	\end{align*}
	by Lemma \ref{lemmeascap}. Letting $\epsilon\to 0$, $b'_1\to b_1$, and $b'_2\to b_2$, we obtain the desired inequality by the right continuity of $F_f^{h_i}$.\par
	
	The case $N=1$ follows from the case $N=2$ by taking $h_2\equiv 0$ and $b_2=-1$.\par
	
	We prove the case $N>2$ by induction. Since
	\[
	\{h_k\le b_k\}\cup\{h_{k+1}\le b_{k+1}\}=\{\deco{w}\le0\}
	\quad\text{for } w:=(h_k-b_k)\wedge(h_{k+1}-b_{k+1}),
	\]
	we conclude that
	\[
	F_f^g(a)\le\sum_{i=1}^{N-1}F_f^{h_i}(b_i)+F_f^{\deco{w}}(0)\le\sum_{i=1}^{N}F_f^{h_i}(b_i),
	\]
	by the induction hypothesis.
\end{proof}

\begin{defi}\label{defOM}
	For $f\in\core$, we define
	\begin{gather*}
		\mu_{\langle f \rangle}(U)=\sup\{F_f^g(a)\mid g\in\core,\ a<0,\ \{g\le a\}\subset U\}
		\intertext{for an open set $U\subset X$, and}
		\mu_{\langle f \rangle}(A)=\inf\{\mu_{\langle f \rangle}(U)\mid A\subset U,\ U\text{ is open}\}
	\end{gather*}
	for any $A\subset X$.
\end{defi}

\begin{thm}\label{thmMOM}
	$\mu_{\langle f \rangle}$ is a metric outer measure; that is, $\mu_{\langle f \rangle}$ is an outer measure satisfying
	$\mu_{\langle f \rangle}(A\cup B)=\mu_{\langle f \rangle}(A)+\mu_{\langle f \rangle}(B)$
	for $A,B\subset X$ with $d(A,B)>0$.
\end{thm}

\begin{proof}
	We have $\{g\le a\}=\emptyset$ if and only if $a<\min_{x\in X} g(x)$, which implies $\mu_{\langle f \rangle}(\emptyset)=0$.
	Next, we prove countable subadditivity for open sets.
	Let $U$ and $U_i\ (i\ge1)$ be open sets with $U\subset \cup_{i\ge 1}U_i$.
	Fix $g\in\core$ and $a<0$ such that $\{g\le a\}\subset U$ (such $g,a$ exist even if $U=\emptyset$).
	Let
	\[
	\Mf{C}_i=\{B_d(x,r)\mid x\in U_i,\ \ol{B_d(x,2r)}\text{ is compact, and }\ol{B_d(x,2r)}\subset U_i\}
	\]
	for $i\ge1$, and let $\Mf{C}_F$ be a finite subcover of $\{g\le a\}$ chosen from $\cup_{i\ge1}\Mf{C}_i$.
	We may assume that $\Mf{C}_i\cap\Mf{C}_F\ne\emptyset\ (1\le i\le N)$ and $\Mf{C}_F\subset\cup_{i=1}^N \Mf{C}_i$ for some $N\ge1$.
	Define
	\begin{align*}
		E_i&=\cup\{\ol{B_d(x,r)}\mid B_d(x,r)\in \Mf{C}_i\cap\Mf{C}_F\},\\
		V_i&=\cup\{\ol{B_d(x,2r)}\mid B_d(x,r)\in \Mf{C}_i\cap\Mf{C}_F\},\\
		\tau_i(y)&=\max\left\{\left(\left(2-\frac{d(x,y)}{r}\right)\wedge1\right)^+\ \bigg\vert\ B_d(x,r)\in\Mf{C}_i\cap\Mf{C}_F\right\}.
	\end{align*}
	Note that $\tau_i\in C_c(X)$, $\supp(\tau_i)\subset V_i$, and $\tau_i|_{E_i}=1$ for any $i$.
	By \ref{defreg}, there exists $h_i\in\core$ with $\|h_i-\tau_i\|_{\infty}<1/2$, and thus
	\[
	E_i\subset \left\{-h_i\le -\frac{1}{2}\right\}\subset V_i\subset U_i
	\quad\text{and}\quad
	\{g\le a\}\subset\cup_{i=1}^N E_i\subset \cup_{i=1}^N\left\{-h_i\le-\frac{1}{2}\right\}.
	\]
	Hence, 
	$
	F_f^g(a)\le \sum_{i=1}^N F_f^{-h_i}(-1/2)\le\sum_{i=1}^N\mu_{\langle f \rangle}(U_i)
	$ by Lemma \ref{lemFcover},
	and therefore $\mu_{\langle f \rangle}(U)\le\sum_{i=1}^\infty \mu_{\langle f \rangle}(U_i)$.
	Subadditivity for general sets follows from that for open sets.\par
	
	We next show that $\mu_{\langle f \rangle}$ is a metric outer measure.
	Let $U_1,U_2$ be open sets satisfying $d(U_1,U_2)>0$. Take any $g,h\in\core$ and $a,b<0$, such that $\{g\le a\}\subset U_1$ and $\{h\le b\}\subset U_2$.
	By the decreasing compact sets argument, there exist $a'\in(a,0)$ and $b'\in(b,0)$ with $\{g\le a'\}\cap \{h\le b'\}=\emptyset$.
	Let
	\[
	s=(a'-a)\wedge(b'-b)
	\quad\text{and}\quad
	u=\left((g-a-s)\wedge0 \right)-\left((h-b-s)\wedge0\right).
	\]
	Then $u\in\core$, $\{u\le -s\}=\{g\le a\}$, and $\{-u\le -s\}=\{h\le b\}$.
	By definition of $S_n$, we have $f_n^{-|u|,-s}=f_n^{u,-s}+f_n^{-u,-s}$ for sufficiently large $n$.
	Thus,
	\[
	\mu_{\langle f \rangle}(U_1\cup U_2)\ge F^{-|u|}_f(-s)=F^{u}_f(-s)+F^{-u}_f(-s)=F^g_f(a)+F^h_f(b),
	\]
	by \ref{defstrloc} and Lemma \ref{lemFcover}.
	This implies $\mu_{\langle f \rangle}(U_1\cup U_2)\ge\mu_{\langle f \rangle}(U_1)+\mu_{\langle f \rangle}(U_2)$.
	
	Now let $A,B\in 2^X$ satisfy $d(A,B)>0$.
	For any open set $U$ with $A\cup B\subset U$,
	\begin{align*}
		&\mu_{\langle f\rangle}(U)\\
		\ge&\mu_{\langle f\rangle}\bigl(U\cap\{x\mid d(x,A)<d(A,B)/3 \}\bigr)
		+\mu_{\langle f\rangle}\bigl(U\cap\{x\mid d(x,B)<d(A,B)/3\}\bigr)\\
		\ge&\ \mu_{\langle f \rangle}(A)+\mu_{\langle f \rangle}(B).
	\end{align*}
	This concludes the proof.
\end{proof}

\begin{cor}\label{Cormeas}
	$\mu_{\langle f \rangle}|_{\Mc{B}(X)}$ is a Radon measure.
\end{cor}

\begin{proof}
	It is known that a metric outer measure is Borel measurable (this follows from Carath\'eodory's criterion; see, e.g., \cite[Section 2.3.2]{Fed}). The other properties are straightforward.
\end{proof}

The remainder of this section is devoted to verifying \ref{condmeasene} for the above construction and to proving \eqref{eqlimmeas}.

\begin{lem}\label{lemdistabs}
	$F_f^g(a)+\lim_{t\uparrow -a}F_f^{-g}(t)=\Mc{E}(f)$ for any $a\in\Mb{R}$, $f\in\core$, and $g\in\Mc{V}$.
\end{lem}

\begin{proof}
	Let $\epsilon>0$ and choose $n\in\Mb{N}$ such that $2^{-n+1}<\epsilon$. Then
	\begin{gather*}
		T_n\circ f=(f_n^{g,a+\epsilon}+ f_n^{-g,-a-2\epsilon})\vee \left((T_n\circ f)\wedge (S_n^{a+2\epsilon}\circ g)\wedge (S_n^{-a-\epsilon}\circ(-g)) \right),\\
		f_n^{g,a+\epsilon}+ f_n^{-g,-a-2\epsilon}=(T_n\circ f)\wedge \left((S_{n}^{a+\epsilon}\circ g)+ (S_n^{-a-2\epsilon}\circ (-g))\right).
	\end{gather*}
	From \ref{defstrloc}, Corollary \ref{corincl}~(2), Lemma \ref{lemmeascap}, and the above identities, we obtain
	\begin{gather*}
		\lim_{n\to\infty}\Mc{E}(T_n\circ f)\le \left(\left(F_f^g(a+\epsilon)+F_f^{-g}(-a-2\epsilon)\right)^{\frac{1}{p}}+\left(F_f^g(a+2\epsilon)-F_f^g(a)\right)^{\frac{1}{p}}\right)^p,\\
		F_f^g(a+\epsilon)+F_f^{-g}(-a-2\epsilon)\le\lim_{n\to\infty}\Mc{E}(T_n\circ f).
	\end{gather*}
	Since $\Mc{E}(T_n\circ f)=\Mc{E}(f)\ (n\in\Mb{N})$ by Lemma \ref{lemfold}, letting $\epsilon\downarrow 0$ we conclude that
	\[
	F_f^g(a)+\lim_{t\uparrow -a}F_f^{-g}(t)=\Mc{E}(f).
	\]
\end{proof}

\begin{prop}\label{propmeasrep}
	For any $f\in\core$ and $g\in\Mc{V}$,
	\begin{enumerate}
		\item $\mu_{\langle f \rangle}(X)=\Mc{E}(f)$ and $\mu_{\langle f\rangle}(\{f=0\})=0$.
		\item $\mu_{\langle f \rangle}(\{g\le a\})=F^g_f(a)$ for any $a\in\Mb{R}$.
	\end{enumerate}
\end{prop}

\begin{proof}
	We may assume $g\in\core$. We first prove (2) for $a<0$.
	The inequality $F_f^g(a)\le \mu_{\langle f \rangle}(\{g\le a\})$ follows directly from the definition of $\mu_{\langle f\rangle}$.
	For any $a'\in(a,0)$,
	\[
	\mu_{\langle f \rangle}(\{g\le a\})\le\mu_{\langle f \rangle}(\{g< a'\})\le F_f^g(a'),
	\]
	by Lemma \ref{lemFcover}. Letting $a'\downarrow a$, we obtain $\mu_{\langle f \rangle}(\{g\le a\})=F^g_f(a)$.\par
	
	Next, we prove (1). Note that $(f_n^{g,-a}+ f_{n}^{-g,-a})\vee f_{n}^{|g|,a}=T_n\circ f$ for any $n,a$ with $2^{-n}<a$.
	Set $g=f$, then $f_n^{|f|,a}=\left(T_n\wedge (S_n^a\circ|\cdot|) \right)\circ f$.
	Hence, $\Mc{E}(f_n^{|f|,a})\le\Mc{E}(\eta_{(-2a,2a)}\circ f)$ for $n$ with $2^{-n}\le a$ by Corollary \ref{corfold}.
	Therefore,
	\begin{align*}
		\Mc{E}(f)
		&\le\liminf_{n\to\infty}\left(\left(\Mc{E}(f_n^{f,-a})+\Mc{E}(f_n^{-f,-a})\right)^{\frac{1}{p}}+\Mc{E}(f_n^{|f|,a})^{\frac{1}{p}} \right)^p\\
		&\le\left(\left(\mu_{\langle f\rangle}(\{f\le-a\})+\mu_{\langle f\rangle}(\{-f\le-a\})\right)^{\frac{1}{p}}
		+\Mc{E}(\eta_{(-2a,2a)}\circ f)^{\frac{1}{p}}\right)^p.
	\end{align*}
	Since $\Mc{E}(\eta_{(-2a_3,2a_3)}\circ f)\le \Mc{E}\left(\left(\frac{\eta_{(-2a_1,2a_1)}+\eta_{(-2a_2,2a_2)} }{2}\right)\circ f \right)$ for $0<a_3<a_2<a_1$ by Corollary \ref{corfold},
	and since $|\eta_{(-2a,2a)}\circ f|\le |f|$ and $\eta_{(-2a,2a)}\circ f\to0$ on $X$ as $a\downarrow0$,
	the dominated convergence theorem and Lemma \ref{lemconvconv} imply $\lim_{a\downarrow 0}\Mc{E}(\eta_{(-2a,2a)}\circ f)=0$.
	Hence,
	\[
	\Mc{E}(f)\le \mu_{\langle f\rangle}(\{f\ne 0\})=\lim_{a\downarrow 0}\mu_{\langle f\rangle}(\{|f|\ge a\}),
	\]
	which proves (1), because $\mu_{\langle f\rangle}(X)\le\Mc{E}(f)$ by definition.\par
	
	We now prove (2) for $a\ge0$.
	By (1), (2) for $a<0$, and Lemma \ref{lemdistabs}, we obtain
	\[
	\mu_{\langle f \rangle}(\{g\le a\})
	=\Mc{E}(f)-\lim_{t\uparrow -a}\mu_{\langle f \rangle}(\{-g\le t\})
	=\Mc{E}(f)-\lim_{t\uparrow -a}F_f^{-g}(t)
	=F^g_f(a).
	\]
\end{proof}

\section{\texorpdfstring{Proof of Theorem \ref{thmunique} (Part III): triangle inequality and $p$-Clarkson inequalities for energy measures}{Triangle inequality and p-Clarkson inequalities for energy measures}}\label{secineqs}
In this section, we prove \ref{condmeasloc}--\ref{condmeasclk} and define $\mu_{\langle f\rangle}$ for any $f\in\Mc{F}$.
In particular, we establish the triangle inequality and the $p$-Clarkson inequalities.
The main obstacle is that $T_n\circ(f+g)\ne T_n\circ f+T_n\circ g$.
To overcome this issue, we use a purely arithmetic decomposition of the dyadic truncations.
For technical reasons, we first prove \ref{condmeasclk} on $\core$, excluding the inequality \eqref{eqCLK1}.

\begin{lem}\label{lemforclk1}
	For any $f\in\core$ and $a\in\Mb{R}$,
	\begin{enumerate}
		\item $\mu_{\langle af\rangle}=|a|^p\,\mu_{\langle f\rangle}$.
		\item $\mu_{\langle |f-a|-|a|\rangle}=\mu_{\langle f\rangle}$.
	\end{enumerate}
\end{lem}

\begin{proof}
	\noindent(1)
	Since $T_n\circ f=T_n\circ(-f)$, we have $\mu_{\langle f\rangle}=\mu_{\langle -f\rangle}$; hence we may assume $a>0$.
	Fix $b,b'\in\Mb{R}$ with $b<b'$ and $g\in\Mc{V}$.
	Choose $m>n$ so that $2^{-m}<b'-b$ and $2^{m}a\ge 2^{n+1}$. 
	We define
	\begin{align*}
		I_{n,m}&=
		\left(\Cup_{i=1}^{2^{n-1}N}\left( \frac{2i-1}{2^n},\left\lceil \frac{2i-1}{2^n}\right\rceil_{\frac{1}{2^m a}} \right)\right)
		\Cup
		\left(\Cup_{i=1}^{2^{n-1}N}\left(\left\lfloor \frac{2i}{2^n}\right\rfloor_{\frac{1}{2^ma}},\frac{2i}{2^n}\right)\right),\\
		J_{n,m}&=
		\left(\Cup_{i=1}^{2^{n-1}N}\left( \frac{2i-1}{2^n},\frac{2i-1}{2^n}+\frac{1}{2^m a} \right)\right)
		\Cup
		\left(\Cup_{i=1}^{2^{n-1}N}\left( \frac{2i}{2^n }-\frac{1}{2^m a},\frac{2i}{2^n}\right)\right),\\
		u_{n,m}(x)&=2\int_0^{f(x)}\indi_{I_{n,m}\cup(-I_{n,m})}(t)\,dt,\quad
		v_{n,m}(x)=2\int_0^{f(x)}\indi_{J_{n,m}\cup(-J_{n,m})}(t)\,dt,
	\end{align*}
	where
	\[
	\lfloor x\rfloor_\alpha=\max\{\alpha i\mid i\in\Mb{Z},\ \alpha i\le x \},
	\qquad
	\lceil x \rceil_\alpha=\min\{\alpha i\mid i\in\Mb{Z},\ \alpha i\ge x\},
	\]
	and $N=\left\lceil \|f\|_\infty\right\rceil_1$. Then, by construction of $u_{n,m}$, we have $T_m\circ(af)=T_m\bigl(a(T_n\circ f)+u_{n,m}\bigr)$.
	Moreover, since $f_n^{g,b'}=T_n\circ f$ on $\{g\le b+2^{-m}\}$ (because $b+2^{-m}<b'$), we have
	\[
	(af)_m^{g,b}=\bigl(T_m\circ(a f_n^{g,b'}+u_{n,m})\bigr)\wedge(S_m^b\circ g).
	\]
	
	By Corollary \ref{corfold}, the sequence $\{v_{n,m}\}_{m>n}$ satisfies
	\[
	\Mc{E}(v_{n,m_3})\le\Mc{E}\Bigl(\frac{v_{n,m_1}+v_{n,m_2}}{2}\Bigr)\qquad(m_3\ge m_2\ge m_1>n).
	\]
	Since $|v_{n,m}|\le |f|$ and $\|v_{n,m}\|_\infty\le 2^{n+1-m}N\to 0$ as $m\to\infty$, the dominated convergence theorem and Lemma \ref{lemconvconv} yield
	$
	\lim_{m\to\infty}\Mc{E}(v_{n,m})=0$.
	Hence $\lim_{m\to\infty}\Mc{E}(u_{n,m})=0$ by Corollary \ref{corfold}.
	Therefore,
	\[
	F_{af}^g(b)
	\le \lim_{m\to\infty}\left(a\,\Mc{E}(f_n^{g,b'})^{1/p}+\Mc{E}(u_{n,m})^{1/p}+\Mc{E}^\Mc{V}(S_m^b\circ g)^{1/p}\right)^p
	= a^p\,\Mc{E}(f_n^{g,b'}).
	\]
	Letting $n\to\infty$ and then $b'\downarrow b$, we obtain
	$F_{af}^g(b)\le a^p F_f^g(b)$ by the right continuity of $F_f^g$.
	Similarly, $F_f^g(b)\le a^{-p}F_{af}^g(b)$.
	Since $g$ and $b$ are arbitrary, we conclude that $\mu_{\langle af\rangle}=a^p\mu_{\langle f\rangle}$.
	
	\noindent(2)
	By (1), it suffices to consider $a=0$ and $a=1$.
	In either case, we have $T_n\circ(|f-a|-|a|)=T_n\circ f$ for every $n$, and the claim follows.
\end{proof}

\begin{thm}\label{thmclk}
	For any $A\in\Mc{B}(X)$, the $p$-Clarkson inequalities for $f\mapsto\mu_{\langle f\rangle}(A)^{1/p}$ hold on $\core$.
\end{thm}

\begin{proof}[Proof of Theorem \ref{thmclk} excluding \eqref{eqCLK1}]
	We first prove \eqref{eqCLK3} in the case $A=\{h\le a\}$ for some $h\in\core$ and $a<0$.
	Fix $b>a$.
	Let $n\in\Mb{N}$ satisfy $2^{-n}<b-a$, and choose $\epsilon\in(0,2^{-(n+1)}\wedge(-a))$.
	For $k=0,1$, define
	\begin{gather*}
		u_{n,\epsilon,k}:=\sum_{\substack{i,j\in\Mb{Z}:\\ i+j\equiv k\ (\mathrm{mod}\ 2)}}\left(\frac{1}{2^{n+1}}-\left|f-\frac{i}{2^n}-\frac{1}{2^{n+1}}-\epsilon\right|\vee\left|g-\frac{j}{2^n}-\frac{1}{2^{n+1}}-\epsilon\right|\right)^+,\\
		v_{n,\epsilon,k}:=\left(\left(h-a-\epsilon\right)\vee (-u_{n,\epsilon,k}) \right)\wedge 0.
	\end{gather*}
	Note that
	\[
	\left\{u_{n,\epsilon,k}\ge \epsilon \right\}
	=\Cup_{\substack{i,j\in\Mb{Z}:\\ i+j\equiv k\ (\mathrm{mod}\ 2)}}\left\{\frac{i}{2^n}+2\epsilon\le f\le\frac{i+1}{2^n},\ \frac{j}{2^n}+2\epsilon\le g\le \frac{j+1}{2^n}\right\}.
	\]
	Moreover, $v_{n,\epsilon,k}\in\core$ by Corollary \ref{corV}.
	For $m>n$ with $2^{-m}<\epsilon$, define
	\[
	D_{n,m}(t)=2\sum_{i\in\Mb{Z}}(-1)^{i+1}\left(\left(\frac{1}{2^m}-\left|t-\frac{i}{2^n}-\frac{1}{2^m}\right| \right)^+\right).
	\]
	Then $D_{n,m}\circ f,D_{n,m}\circ g\in\core$ by Corollary \ref{corincl}~(3), and
	$
	\lim_{m\to\infty}\Mc{E}(D_{n,m}\circ f)=\lim_{m\to\infty}\Mc{E}(D_{n,m}\circ g)=0$
	by the dominated convergence theorem, Lemma \ref{lemconvconv}, and Corollary \ref{corfold}.
	
	For $k=0,1$, we have
	\[
	T_m\circ\left(T_n\circ f+T_n\circ g+D_{n,m}\circ f+D_{n,m}\circ g\right)=T_m\circ (f+(-1)^k g)
	\]
	on $\{u_{n,\epsilon,k}\ge\epsilon-2^{-m}\}$.
	Also, $T_n\circ f=f_n^{h,b}$ and $T_n\circ g=g_n^{h,b}$ on $\{h\le b\}$, and
	\[
	\{v_{n,\epsilon,k}\le -\epsilon+2^{-m}\}=\{h\le a+2^{-m} \}\cap\{u_{n,\epsilon,k}\ge \epsilon-2^{-m}\}.
	\]
	Hence,
	\begin{align*}
		&\left(f+g\right)_{m}^{(v_{n,\epsilon,0}),-\epsilon}+\left(f-g\right)_{m}^{(v_{n,\epsilon,1}),-\epsilon}\\
		=&T_m\circ \left((f_{n}^{h,b}+g_{n}^{h,b})+D_{n,m}\circ f+D_{n,m}\circ g \right)
		\wedge \left(\sum_{k=0,1} S_{m}^{-\epsilon}\circ v_{n,\epsilon,k}\right),
	\end{align*}
	because the sets $\{v_{n,\epsilon,k}\le -\epsilon+2^{-m}\}$ are disjoint for $k$.
	Therefore,
	\begin{align*}
		&\mu_{\langle f+g \rangle}\left(\{h\le a,\ u_{n,\epsilon,0}\ge\epsilon\}\right)
		+\mu_{\langle f-g \rangle}\left(\{h\le a,\ u_{n,\epsilon,1}\ge\epsilon\}\right)\\
		=&\lim_{m\to\infty}\left(\Mc{E}\left( (f+g)_{m}^{(v_{n,\epsilon,0}),-\epsilon}\right)+\Mc{E}\left((f-g)_m^{(v_{n,\epsilon,1}),-\epsilon} \right)\right)\\
		=&\lim_{m\to\infty}\Mc{E}\left((f+g)_m^{(v_{n,\epsilon,0}),-\epsilon}+(f-g)_m^{(v_{n,\epsilon,1}),-\epsilon} \right)\\
		\le&
		\lim_{m\to\infty}\left(\Mc{E}(f_{n}^{h,b}+g_{n}^{h,b})^{\frac{1}{p}}
		+\!\!\!\sum_{\tau=f,g}\Mc{E}(D_{n,m}\circ \tau)^{\frac{1}{p}}
		+\sum_{k=0,1} \Mc{E}^\Mc{V}\left(\deco{S_{m}^{-\epsilon}\circ v_{n,\epsilon,k}}\right)^{\frac{1}{p}} \right)^p\\
		=&\Mc{E}(f_n^{h,b}+g_n^{h,b}),
	\end{align*}
	by \ref{defstrloc}, Corollary \ref{corfold}, and Corollary \ref{corincl}~(2), and since
	$\{v_{n,\epsilon,k}\le-\epsilon\}=\{h\le a\}\cap \{u_{n,\epsilon,k}\ge\epsilon\}$.
	Letting $\epsilon\downarrow 0$, we obtain
	$
	\mu_{\langle f+g \rangle}(A_{n,a,0})+\mu_{\langle f-g \rangle}(A_{n,a,1})
	\le \Mc{E}(f_{n}^{h,b}+g_{n}^{h,b})$,
	where
	\begin{equation}\label{defanak}
		A_{n,a,k}=\Cup_{\substack{i,j\in\Mb{Z}:\\ i+j\equiv k\ (\mathrm{mod}\ 2)}} \left\{\frac{i}{2^{n}}<f\le\frac{i+1}{2^{n}},\ \frac{j}{2^{n}}<g\le\frac{j+1}{2^{n}},\ h\le a\right\}
	\end{equation}
	for $k=0,1$.
	Similarly,
	$
	\mu_{\langle f-g \rangle}(A_{n,a,0})+\mu_{\langle f+g \rangle}(A_{n,a,1})
	\le \Mc{E}(f_n^{h,b}-g_{n}^{h,b})$.
	Combining these inequalities with \ref{defclk} yields
	\begin{align}
		&\mu_{\langle f+g \rangle}(\{h\le a\})+\mu_{\langle f-g \rangle}(\{h\le a\}) \notag \\
		=&\left(\mu_{\langle f+g \rangle}(A_{n,a,0})+\mu_{\langle f-g \rangle}(A_{n,a,1})\right)
		+\left(\mu_{\langle f-g \rangle}(A_{n,a,0})+\mu_{\langle f+g \rangle}(A_{n,a,1})\right)\label{eqAchilles}\\
		\le&\Mc{E}(f_n^{h,b}+g_n^{h,b})+\Mc{E}(f_n^{h,b}-g_n^{h,b})\notag \\
		\le& 2\left(\Mc{E}(f_n^{h,b})^{\frac{1}{p-1}}+\Mc{E}(g_n^{h,b})^{\frac{1}{p-1}}\right)^{p-1}.\notag
	\end{align}
	Letting $n\to\infty$ and then $b\downarrow a$, we conclude that \eqref{eqCLK3} holds in this case.
	The general case $A\in\Mc{B}(X)$ follows from Lemma \ref{lemapproxopen} and the outer regularity of Radon measures.\par
	The proofs of \eqref{eqCLK2} and \eqref{eqCLK4} are analogous.
	Indeed, \eqref{eqCLK4} is equivalent to
	$
	\mu_{\langle f+g\rangle}(A)+\mu_{\langle f-g\rangle}(A)
	\le2^{p-1}\left(\mu_{\langle f\rangle}(A)+\mu_{\langle g\rangle}(A) \right)$
	by Lemma \ref{lemforclk1}~(1).
\end{proof}

\begin{rem}
	The above argument does \emph{not} apply to \eqref{eqCLK1}, because the identity that would be needed in place of \eqref{eqAchilles} fails.
\end{rem}

Next we prove \eqref{eqCLK1} by establishing \ref{condmeasloc}.

\begin{lem}\label{lemforwloc}
	There exists $c_p>0$ depending only on $p$ such that
	\[
	\mu_{\langle f\vee(g-a)\rangle}(A)\vee\mu_{\langle f\wedge(g+a)\rangle}(A)
	\le c_p\left(\mu_{\langle f \rangle}(A)+\mu_{\langle g \rangle}(A)\right)
	\]
	for any $A\in\Mc{B}(X)$, $f,g\in\core$, and $a\ge0$.
\end{lem}

\begin{proof}
	Since $f\vee(g-a)=\frac{f+g}{2}+\left(\frac{|f-g+a|}{2}-\frac{a}{2}\right)$, we have
	\begin{align*}
		\mu_{\langle f\vee(g-a)\rangle}(A)
		\le&\,2^{(p-1)\vee1}\left(\mu_{\left\langle \frac{f+g}{2} \right\rangle}(A)+\mu_{\left\langle \frac{|f-g+a|}{2}-\frac{a}{2} \right\rangle}(A)\right)\\
		=&\,2^{((p-1)\vee1)-p}\left( \mu_{\langle f+g \rangle}(A)+\mu_{\langle f-g \rangle}(A)\right)\\
		\le&\,2^{|p-2|}\left( \mu_{\langle f \rangle}(A)+\mu_{\langle g \rangle}(A)\right),
	\end{align*}
	by Lemma \ref{lemforclk1}, \eqref{eqCLK2}, and \eqref{eqCLK4}.
	The estimate for $\mu_{\langle f\wedge(g+a)\rangle}(A)$ is proved in the same way.
\end{proof}

\begin{prop}\label{propwloc}
	\ref{condmeasloc} holds. In particular, $\mu_{\langle f \rangle}(A)=0$ if $f|_{A}$ is constant.
\end{prop}

\begin{proof}
	We first consider the case $A=\{h\le a\}$ for some $h\in\Mc{V}$ and $a\in\Mb{R}$.
	By Lemma \ref{lemforclk1}~(1), we may assume that $(f-g)|_{A}=0$ or $(f-g)|_{A}=1$.
	In both cases, $(T_n\circ f-T_n\circ g)|_A=0$.
	Therefore, by Proposition \ref{propmeasrep}~(1),
	\[
	\Mc{E}(f_n^{h,a}-g_n^{h,a})
	=\mu_{\left\langle f_n^{h,a}-g_n^{h,a} \right\rangle}\bigl(\{a<h\le  a+2^{-n}\}\bigr).
	\]
	Since $\mu_{\langle f\rangle}=\mu_{\langle T_n\circ f\rangle}$ by definition, we have
	\begin{align*}
		&\mu_{\left\langle f_n^{h,a}-g_n^{h,a} \right\rangle}\bigl(\{a<h\le  a+2^{-n}\}\bigr)\\
		\le&\,2^{(p-1)\vee1}\left(\mu_{\left\langle f_n^{h,a} \right\rangle}\bigl(\{a<h\le  a+2^{-n}\}\bigr)+\mu_{\left\langle g_n^{h,a} \right\rangle}\bigl(\{a<h\le  a+2^{-n}\}\bigr)\right)\\
		\le&\,\begin{multlined}[t][.9\linewidth]
			c_p2^{(p-1)\vee 1}\Bigl(
			\mu_{\left\langle f \right\rangle}\bigl(\{a<h\le a+2^{-n}\}\bigr)
			+\mu_{\left\langle g \right\rangle}\bigl(\{a<h\le a+2^{-n}\}\bigr)\\
			+2\,\mu_{\bigl\langle\deco{\eta_{(-a-2^{-n},-a)}\circ (-h)}\bigr\rangle}\bigl(\{a<h\le a+2^{-n}\}\bigr)
			\Bigr),
		\end{multlined}
	\end{align*}
	by \eqref{eqCLK2}, \eqref{eqCLK4}, and Lemma \ref{lemforwloc}.
	Consequently,
	\begin{gather*}
		\lim_{n\to\infty}\Mc{E}(f_n^{h,a}-g_n^{h,a})
		\le c_p2^{(p-1)\vee1}\left(2\lim_{n\to\infty}\Mc{E}^\Mc{V}\bigl(\deco{S_{n}^a\circ h}\bigr)\right)=0,
		\shortintertext{and hence}
		\mu_{\left\langle f \right\rangle}(\{h\le a\})
		=\lim_{n\to\infty}\Mc{E}(f_n^{h,a})
		=\lim_{n\to\infty}\Mc{E}(g_n^{h,a})
		=\mu_{\left\langle g \right\rangle}(\{h\le a\}).
	\end{gather*}
	
	For the general case, note that the desired statement is equivalent to
	\begin{equation}\label{eqmeaswloc}
		\mu_{\langle f\rangle}|_{\{f-g=c\}}(A)=\mu_{\langle g\rangle}|_{\{f-g=c\}}(A)
	\end{equation}
	for any $c\in\Mb{R}$ and $A\in\Mc{B}(X)$.
	If $A=\{h\le a\}$ for some $h\in\Mc{V}$ and $a\in\Mb{R}$, then \eqref{eqmeaswloc} follows from the previous case, because
	$
	\{h\le a\}\cap\{f-g=c\}=\{(h-a)\vee|f-g-c|\le 0\}$.
	Therefore, \eqref{eqmeaswloc} holds for any $A\in\Mc{B}(X)$ (and any $c\in\Mb{R}$) by Lemma \ref{lemapproxopen} and the outer regularity of Radon measures.
\end{proof}

\begin{proof}[Proof of Theorem \ref{thmclk} for \eqref{eqCLK1}]
	As in the proofs of \eqref{eqCLK2}--\eqref{eqCLK4}, it suffices to consider the case $A=\{h\le a\}$ for some $h\in\Mc{V}$ and $a\in\Mb{R}$.
	Since $(f_n^{h,a}+g_n^{h,a})|_{\{h\ge a+2^{-n}\}}=0$, we have
	\begin{align*}
		&\mu_{\left\langle f_n^{h,a}+g_n^{h,a} \right\rangle}\bigl(\{f_n^{h,a}+g_n^{h,a}=T_n\circ f+T_n\circ g\}\bigr)\\
		=&\mu_{\left\langle f_n^{h,a}+g_n^{h,a} \right\rangle}\bigl(\{f_n^{h,a}+g_n^{h,a}=T_n\circ f+T_n\circ g,\ h\le a+2^{-n}\}\bigr)\\
		=&\mu_{\left\langle T_n\circ f+T_n\circ g \right\rangle}\bigl(\{f_n^{h,a}+g_n^{h,a}=T_n\circ f+T_n\circ g,\ h\le a+2^{-n}\}\bigr)\\
		\le&\mu_{\left\langle T_n\circ f+T_n\circ g \right\rangle}\bigl(\{h\le a+2^{-n}\}\bigr)
		=\mu_{\left\langle f+g \right\rangle}(A_{n,a+2^{-n},0})+\mu_{\left\langle f-g \right\rangle}(A_{n,a+2^{-n},1}),
	\end{align*}
	by Proposition \ref{propwloc}, where $A_{n,a,k}$ is defined by \eqref{defanak}.
	The same argument applies to $f_n^{h,a}-g_n^{h,a}$.
	Hence,
	\begin{align}
		&2\left(\Mc{E}(f_n^{h,a})^{\frac{1}{p-1}}+\Mc{E}(g_n^{h,a})^{\frac{1}{p-1}}\right)^{p-1} \notag\\
		\le&\Mc{E}(f_n^{h,a}+g_n^{h,a})+\Mc{E}(f_n^{h,a}-g_n^{h,a})\notag\\
		\le&\begin{multlined}[t][.9\linewidth]
			\sum_{k=0,1}\left(\mu_{\left\langle f+g \right\rangle}(A_{n,a+2^{-n},k})+\mu_{\left\langle f-g \right\rangle}(A_{n,a+2^{-n},k})\right)\\
			+\mu_{\left\langle f_n^{h,a}+g_n^{h,a} \right\rangle}\bigl(\{f_n^{h,a}+g_n^{h,a}\ne T_n\circ f+T_n\circ g\}\bigr)\\
			+\mu_{\left\langle f_n^{h,a}-g_n^{h,a} \right\rangle}\bigl(\{f_n^{h,a}-g_n^{h,a}\ne T_n\circ f-T_n\circ g\}\bigr),
		\end{multlined}\label{eqforCLK1}
	\end{align}
	by \ref{defclk}.
	Since $h>a$ whenever $f_n^{h,a}\ne T_n\circ f$ or $g_n^{h,a}\ne T_n\circ g$, we have
	\begin{align*}
		&\lim_{n\to\infty}\mu_{\left\langle f_n^{h,a}+g_n^{h,a} \right\rangle}\bigl(\{f_n^{h,a}+g_n^{h,a}\ne T_n\circ f+T_n\circ g\}\bigr)\\
		\le&\lim_{n\to\infty}\mu_{\left\langle f_n^{h,a}+g_n^{h,a} \right\rangle}\bigl(\{a<h\le a+2^{-n}\}\bigr)\\
		\le&\lim_{n\to\infty}2\left(\mu_{\left\langle f_n^{h,a} \right\rangle}\bigl(\{a<h\le a+2^{-n}\}\bigr)+\mu_{\left\langle g_n^{h,a} \right\rangle}\bigl(\{a<h\le a+2^{-n}\}\bigr)\right)\\
		\le&\lim_{n\to\infty}2c_p\Bigl(\mu_{\left\langle f \right\rangle}\bigl(\{a<h\le a+2^{-n}\}\bigr)+\mu_{\left\langle g \right\rangle}\bigl(\{a<h\le a+2^{-n}\}\bigr)
		+2\Mc{E}^\Mc{V}\bigl(\deco{S_n^a\circ h}\bigr)\Bigr)\\
		=&0,
	\end{align*}
	and similarly
	$
	\lim_{n\to\infty}\mu_{\left\langle f_n^{h,a}-g_n^{h,a} \right\rangle}\bigl(\{f_n^{h,a}-g_n^{h,a}\ne T_n\circ f-T_n\circ g\}\bigr)=0$.
	Therefore, \eqref{eqCLK1} follows from \eqref{eqforCLK1}.
\end{proof}
In the remainder of this section, we prove the triangle inequality.
\begin{lem}\label{lemminconti}
	For any $f,g\in\core$ and $a\in\Mb{R}$,
	\[
	\lim_{t\to 0}\Mc{E}^\Mc{V}\Bigl(\deco{(f+tg)\wedge a}-\deco{f\wedge a}\Bigr)=0.
	\]
\end{lem}
\begin{proof}
	We may assume $t\ge 0$.
	By Lemma \ref{lemforclk1}, Theorem \ref{thmclk}, and Proposition \ref{propwloc}, we have
	\begin{align*}
		&\Mc{E}^\Mc{V}\Bigl(\deco{(f+tg)\wedge a}-\deco{f\wedge a}\Bigr)\\
		=&\begin{multlined}[t][.9\linewidth]
			\mu_{\langle tg\rangle}\bigl(\{(f+tg)\vee f\le a\}\bigr)
			+\mu_{\bigl\langle \deco{f+tg}\bigr\rangle}\bigl(\{f+tg<a<f\}\bigr)\\
			+\mu_{\bigl\langle \deco{-f}\bigr\rangle}\bigl(\{f<a<f+tg\}\bigr)
		\end{multlined}\\
		\le&\begin{multlined}[t][.9\linewidth]
			\Mc{E}(tg)
			+2^{(p-1)\vee 1}\Bigl(\Mc{E}(tg)+\mu_{\langle f\rangle}\bigl(\{f+tg<a<f\}\bigr)\Bigr)\\
			+\mu_{\langle f\rangle}\bigl(\{f<a<f+tg\}\bigr).
		\end{multlined}
	\end{align*}
	Hence, the desired limit holds.
\end{proof}

\begin{lem}\label{lemmeasconti}
	For any $f,g,h\in\core$ and $a\in\Mb{R}$, the map
	$
	t\longmapsto \mu_{\langle f+tg\rangle}(\{h\le a\})$
	is continuous on $\Mb{R}$.
\end{lem}
\begin{proof}
	It suffices to show
	$
	\lim_{t\to 0}\mu_{\langle f+tg\rangle}(\{h\le a\})=\mu_{\langle f\rangle}(\{h\le a\})$.
	For $u\in\core$, set
	\begin{align*}
		A_n(u)
		&:=\bigcup_{i\in\Mb{Z}}
		\left\{\,h-u\le -\frac{i}{2^n},\ 
		2a+\frac{i}{2^n}<h+u\le 2a+\frac{i+2}{2^n}\,\right\},\\
		E_n(u)&:=\mu_{\langle h\rangle}\bigl(A_n(u)\bigr)-\mu_{\langle u\rangle}\bigl(A_n(u)\bigr).
	\end{align*}
	Note that (by construction) $\{h\le a\}\subset A_n(u)\subset \{h\le a+2^{-n}\}$.
	
	Proposition \ref{propwloc} yields, for any $v,w\in\core$ and $\alpha,\beta\in\Mb{R}$,
	\begin{align*}
		&\Mc{E}^\Mc{V}\bigl(\deco{(v\wedge \alpha)}+\deco{(w\wedge \beta)}\bigr)
		-\Mc{E}^\Mc{V}\bigl(\deco{(v\wedge \alpha)}-\deco{(w\wedge \beta)}\bigr)\\
		&\qquad=
		\mu_{\langle v+w\rangle}\bigl(\{u\le \alpha,\ v\le \beta\}\bigr)
		-\mu_{\langle v-w\rangle}\bigl(\{u\le \alpha,\ v\le \beta\}\bigr).
	\end{align*}
	Applying this identity with $(v,w)=\bigl(\frac{h-u}{2},\frac{h+u}{2}\bigr)$ and summing over the relevant indices, we obtain
	\begin{align*}
		E_n(u)
		=\sum_{i=-N}^{N}\begin{multlined}[t][.8\linewidth]
			\Mc{E}^\Mc{V}\left(\left(\deco{\frac{h-u}{2}\wedge\frac{-i}{2^{n+1}}}\right)
			+\left(\deco{\frac{h+u}{2}\wedge\left( a+\frac{i}{2^{n+1}}+\frac{1}{2^n}\right)}\right)\right)\\
			-\Mc{E}^\Mc{V}\left(\left(\deco{\frac{h-u}{2}\wedge\frac{-i}{2^{n+1}}}\right)
			-\left(\deco{\frac{h+u}{2}\wedge\left( a+\frac{i}{2^{n+1}}+\frac{1}{2^n}\right)}\right)\right)\\
			-\Mc{E}^\Mc{V}\left(\left(\deco{\frac{h-u}{2}\wedge\frac{-i}{2^{n+1}}}\right)
			+\left(\deco{\frac{h+u}{2}\wedge\left( a+\frac{i}{2^{n+1}}\right)}\right)\right)\\
			+\Mc{E}^\Mc{V}\left(\left(\deco{\frac{h-u}{2}\wedge\frac{-i}{2^{n+1}}}\right)
			-\left(\deco{\frac{h+u}{2}\wedge\left( a+\frac{i}{2^{n+1}}\right)}\right)\right).
		\end{multlined}
	\end{align*}
	Here one may take
	$
	\infty>N\ge C\,2^n\Bigl((\|u\|_{\infty}+\|h\|_{\infty})\vee 1\Bigr)$,
	where $C>0$ is independent of $a,h,u$.
	Therefore, by Lemma \ref{lemminconti}, we have $\lim_{t\to 0}E_n(f+tg)=E_n(f)$ for each fixed $n$.
	Moreover,
	\begin{align*}
		&-\mu_{\langle h\rangle}(\{a<h\le a+2^{-n}\})\\
		\le\ &\mu_{\langle h\rangle}(\{h\le a\})-\mu_{\langle f+tg\rangle}(\{h\le a\})-E_n(f+tg)\\
		\le\ &\mu_{\langle f+tg\rangle}(\{a<h\le a+2^{-n}\})\\
		\le\ &2^{(p-1)\vee 1}\Bigl(\mu_{\langle f\rangle}(\{a<h\le a+2^{-n}\})
		+|t|^p\,\mu_{\langle g\rangle}(\{a<h\le a+2^{-n}\})\Bigr).
	\end{align*}
	Hence,
	$\lim_{n\to\infty}E_{n}(f+tg)=\mu_{\left\langle h \right\rangle}(\{h\le a\})-\mu_{\left\langle f+tg \right\rangle}(\{h\le a\})$
	uniformly for $t\in[-1,1]$.  Combining this uniform convergence with $\lim_{t\to 0}E_n(f+tg)=E_n(f)$ yields
	$
	\lim_{t\to 0}\mu_{\langle f+tg\rangle}(h\le a)=\mu_{\langle f\rangle}(h\le a)$
	as desired.
\end{proof}

Now we are ready to prove the triangle inequality. For future applications, we prove the following general statement.

\begin{lem}\label{lemSNorm}
	Let $Y$ be a $K$-vector space ($K=\Mb{R}$ or $\Mb{C}$), and let $\Phi:Y\to[0,\infty)$ satisfy:
	\begin{description}
		\Ncond{condhomo} (1-homogeneous) $\Phi(ax)=|a|\Phi(x)$ for any $a\in K$ and $x\in Y$.
		\Ncond{condMax} $\Phi\!\left(\frac{x+y}{2}\right)\le \Phi(x)\vee \Phi(y)$ for any $x,y\in Y$.
		\Ncond{condlsc} For any $x,y\in Y$, the map $t\mapsto \Phi(x+ty)$ is lower semicontinuous on $\Mb{R}$.
	\end{description}
	Then $\Phi$ is a seminorm on $Y$.
\end{lem}
\begin{proof}
	It suffices to prove the triangle inequality. We divide the proof into three cases.
	
	\begin{description}
		\item[Case 1: $\Phi(x),\Phi(y)>0$.]
		By \ref{condhomo},
		\[
		\frac{\Phi(x+y)}{\Phi(x)+\Phi(y)}
		=\Phi\!\left(\frac{\Phi(x)}{\Phi(x)+\Phi(y)}\cdot \frac{x}{\Phi(x)}
		+\frac{\Phi(y)}{\Phi(x)+\Phi(y)}\cdot\frac{y}{\Phi(y)} \right).
		\]
		Hence it is enough to show $\Phi(tx+(1-t)y)\le 1$ for every $t\in(0,1)$ under the normalization $\Phi(x)=\Phi(y)=1$.
		We argue by contradiction.
		
		Assume $\Phi(t_*x+(1-t_*)y)>1$ for some $t_*\in(0,1)$. Define
		\begin{align*}
			a&:=\inf\{t\in(t_*,1]\mid \Phi(tx+(1-t)y)\le1\},\\
			b&:=\sup\{t\in[0,t_*)\mid \Phi(tx+(1-t)y)\le1\}.
		\end{align*}
		Then, by the definition of $a,b$ and $t_*$, we have
		\[
		\Phi\!\left(\frac{a+b}{2}x+\left(1-\frac{a+b}{2}\right)y\right)>1,
		\]
		while $\Phi(ax+(1-a)y)\le 1$ and $\Phi(bx+(1-b)y)\le 1$ by lower semicontinuity \ref{condlsc}.
		This contradicts \ref{condMax}.
		
		\item[Case 2: $\Phi(x)=\Phi(y)=0$.]
		By the same argument as in Case~1, we obtain $\Phi(tx+(1-t)y)=0$ for every $t\in[0,1]$.
		In particular, $\Phi(x+y)=2\,\Phi\!\left(\frac{x+y}{2}\right)=0$.
		
		\item[Case 3: otherwise.]
		We may assume $\Phi(x)>0$, $\Phi(y)=0$, and $\Phi(x+y)>0$.
		By the proof of Case~2 and \ref{condlsc}, there exists $t_0\in[0,1/2)$ such that
		$\Phi(tx+(1-t)y)=0$ for $t\in[0,t_0]$ and $\Phi(tx+(1-t)y)>0$ for $t\in(t_0,1]$.
		
		We claim that $\liminf_{t\downarrow t_0}\Phi(tx+(1-t)y)=0$.
		If not, then there exist $t_1\in(t_0,1/2)$ and $\epsilon>0$ such that
		$\Phi(tx+(1-t)y)\ge \epsilon$ for all $t\in(t_0,t_1]$.
		Let $w=t_0x+(1-t_0)y$ and $z=t_1x+(1-t_1)y$.
		For $s\in(0,1)$, we have
		\[
		\Phi\!\left(\frac{w+sz}{2}\right)
		=\frac{1+s}{2}\,\Phi\!\left(\frac{w+sz}{1+s}\right)\ge \frac{\epsilon}{2},
		\]
		where we used \ref{condhomo}. On the other hand,
		\[
		\Phi\!\left(\frac{w+sz}{2}\right)\le \Phi(sz)\vee\Phi(w)=s\,\Phi(z)
		\]
		by \ref{condhomo} and \ref{condMax}, which is a contradiction for sufficiently small $s$.
		Therefore $\liminf_{t\downarrow t_0}\Phi(tx+(1-t)y)=0$.
		
		Now fix $t\in(t_0,1/2)$. By Case~1 and \ref{condhomo}, applied to the pair $x$ and $tx+(1-t)y$,
		\[
		\Phi\!\left(\frac{x+y}{2}\right)
		\le \frac{1-2t}{2-2t}\Phi(x)+\frac{1}{2-2t}\Phi(tx+(1-t)y).
		\]
		Letting $t\downarrow t_0$ and using $\liminf_{t\downarrow t_0}\Phi(tx+(1-t)y)=0$, we obtain
		\[
		\Phi(x+y)
		=2\Phi\!\left(\frac{x+y}{2}\right)
		\le 2\cdot\frac{1-2t_0}{2-2t_0}\Phi(x)
		\le \Phi(x)
		=\Phi(x)+\Phi(y).
		\]
	\end{description}
\end{proof}

\begin{rem}
	By a similar argument, the same conclusion holds if we replace \ref{condMax} and \ref{condlsc} with the following two conditions:
	\begin{description}
		\Ncond{condMconvex} If $\Phi(x)=\Phi(y)=1$, then $\Phi\!\left(\frac{x+y}{2}\right)\le 1$.
		\Ncond{condconti} For any $x,y\in Y$, the map $t\mapsto \Phi(x+ty)$ is continuous on $\Mb{R}$.
	\end{description}
\end{rem}

\begin{thm}\label{thmtri}
	For any $A\in\Mc{B}(X)$, the functional $f\mapsto \mu_{\langle f\rangle}(A)^{1/p}$ is a seminorm on $\core$.
\end{thm}
\begin{proof}
	As in the proof of Theorem \ref{thmclk}, it suffices to consider the case $A=\{h\le a\}$ for some $h\in\Mc{V}$ and $a\in\Mb{R}$.
	Since inequalities \eqref{eqCLK1} or \eqref{eqCLK4} imply \ref{condMax}, the claim follows from Lemmas \ref{lemforclk1}, \ref{lemmeasconti}, and \ref{lemSNorm}, together with Theorem \ref{thmclk}.
\end{proof}

Theorem \ref{thmtri} enables us to define $\mu_{\langle u\rangle}$ for any $u\in\Mc{F}$.

\begin{defi}
	Let $u\in\Mc{F}$. By \ref{defreg}, there exists a sequence $u_n\in\core$ $(n\ge1)$ such that $\Mc{E}(u-u_n)\to 0$ as $n\to\infty$.
	We define, for $A\in\Mc{B}(X)$,
	\[
	\mu_{\langle u\rangle}(A):=\lim_{n\to\infty}\mu_{\langle u_n\rangle}(A).
	\]
\end{defi}

\begin{rem}
	Since $|\mu_{\left\langle f \right\rangle}(A)^\frac{1}{p}-\mu_{\left\langle g \right\rangle}(A)^\frac{1}{p}|\le\Mc{E}(f-g)^\frac{1}{p}$, $\mu_{\left\langle u \right\rangle}(A)$ is well-defined.
\end{rem}

\begin{cor}\label{corexist}
	$\{\mu_{\langle f\rangle}\}_{f\in\Mc{F}}$ is a family of Radon measures satisfying \ref{condmeasene}--\ref{condmeasclk}.
\end{cor}
The proof is straightforward.

\section{\texorpdfstring{Proof of Theorem \ref{thmunique} (Part IV): chain and Leibniz rules for energy measures}{Chain and Leibniz rules for energy measures}}\label{secchain}
In this section, we prove the chain and Leibniz rules for energy measures. We first recall the differentiability of the map $t\mapsto \mu_{\langle f+tg\rangle}(A)$, which appears in \ref{condmeaschain}.

\begin{defi}[{\cite[Proposition 4.3, Definition 4.4, and Theorem 4.6]{KSdiff}}]
	For any $f,g\in\Mc{F}$ and $A\in\Mc{B}(X)$, the derivative
	\[
	\mu_{\langle f;g\rangle}(A):=\frac{1}{p}\frac{\der}{\der t}\mu_{\langle f+tg\rangle}(A)\Big|_{t=0}
	\]
	exists. Moreover, $\mu_{\langle f;g\rangle}$ is a signed Borel measure on $(X,d)$ with $\mu_{\langle f;f\rangle}=\mu_{\langle f\rangle}$.
\end{defi}

\begin{lem}\label{lemsmwloc}
	Let $f,g,h\in\core$ and $A\in\Mc{B}(X)$.
	\begin{enumerate}
		\item $\mu_{\langle f;g\rangle}(A)=\mu_{\langle g;f\rangle}(A)=0$ if $f|_{A}$ is constant.
		\item $\mu_{\langle f;h\rangle}(A)=\mu_{\langle g;h\rangle}(A)$ and $\mu_{\langle h;f\rangle}(A)=\mu_{\langle h;g\rangle}(A)$ if $(f-g)|_{A}$ is constant.
	\end{enumerate}
\end{lem}
\begin{proof}
	This follows immediately from \ref{condmeasloc} and \cite[Theorem 4.5]{KSdiff}.
\end{proof}

To discuss Radon--Nikodym derivatives of energy measures, we need the following notion, introduced in \cite{Hin10} for $p=2$ and extended to $p>1$ in \cite{MS}.

\begin{defi}[Minimal energy-dominant measure]
	A $\sigma$-finite Borel measure $\mu$ on $(X,d)$ is called a \emph{minimal energy-dominant measure} of $(\Mc{E},\Mc{F})$ if the following two conditions hold:
	\begin{description}
		\DMcond{conddomination} $\mu_{\langle f\rangle}\ll \mu$ for any $f\in\Mc{F}$.
		\DMcond{condminimality} If $\tau$ is a $\sigma$-finite Borel measure satisfying \ref{conddomination}, then $\mu\ll \tau$.
	\end{description}
\end{defi}

Note that $(\Mc{F},\|\cdot\|_{\Mc{E}_1})$ is separable by \cite[Corollary 3.16]{KSdiff}.
\begin{lem}\label{lemMED}
	Let $\{f_i\}_{i\ge 1}$ be a dense subset of $\Mc{F}$, and let $\{a_i\}_{i\ge 1}$ be a sequence of positive numbers such that $\sum_{i\ge 1}a_i\Mc{E}(f_i)<\infty$. Then
	$
	\sum_{i\ge 1} a_i\,\mu_{\langle f_i\rangle}$
	is a finite minimal energy-dominant measure. In particular,
	$
	\sum_{i\ge 1} 2^{-i}\,\Mc{E}(f_i)^{-1}\,\mu_{\langle f_i\rangle}$ is a finite minimal energy-dominant measure if we choose $f_i$ with $\Mc{E}(f_i)>0$.
\end{lem}
\begin{proof}
	The same argument as in \cite[Lemma 9.20(b)]{MS} applies (the ``self-similarity'' assumption in \cite[Theorem 1.1]{MS} was not used there).
\end{proof}

We fix a finite minimal energy-dominant measure $\mu$ of $(\Mc{E},\Mc{F})$ for the rest of this section.
Note that $\mu_{\langle f;g\rangle}\ll \mu$ by \cite[(4.8) in Theorem 4.5]{KSdiff}.
We denote the Radon--Nikodym derivative of $\mu_{\langle f\rangle}$ (resp. $\mu_{\langle f;g\rangle}$) with respect to $\mu$ by $\frac{\der \mu_{\langle f\rangle}}{\der \mu}$ (resp. $\frac{\der \mu_{\langle f;g\rangle}}{\der \mu}$).
By \cite[Theorem 4.5]{KSdiff} (and the uniqueness of Radon--Nikodym derivatives), we have
\begin{equation}\label{eqemlinear}
	\frac{\der\mu_{\langle f;ag+bh\rangle}}{\der\mu}
	=
	a\,\frac{\der\mu_{\langle f;g\rangle}}{\der\mu}
	+b\,\frac{\der\mu_{\langle f;h\rangle}}{\der\mu}
	\qquad \mu\text{-a.e.},
\end{equation}
for any $a,b\in\Mb{R}$ and $f,g,h\in\Mc{F}$.

\begin{lem}\label{lemconti}
	For any $f\in\Mc{F}$, the maps
	\[
	(\Mc{F},\|\cdot\|_{\Mc{E}_1})\ni g\longmapsto \frac{\der\mu_{\langle f;g\rangle}}{\der\mu}\in L^1(\mu)
	\quad\text{and}\quad
	(\Mc{F},\|\cdot\|_{\Mc{E}_1})\ni g\longmapsto \frac{\der\mu_{\langle g;f\rangle}}{\der\mu}\in L^1(\mu)
	\]
	are continuous.
\end{lem}
\begin{proof}
	Fix $g_1,g_2,f\in\Mc{F}$ and set
	\[
	A:=\left\{\frac{\der\mu_{\langle g_1;f\rangle}}{\der\mu}-\frac{\der\mu_{\langle g_2;f\rangle}}{\der\mu}\ge 0\right\}.
	\]
	Then, by \cite[Theorem 4.5]{KSdiff},
	\begin{align*}
		&\int_X \left|\frac{\der\mu_{\langle g_1;f\rangle}}{\der\mu}-\frac{\der\mu_{\langle g_2;f\rangle}}{\der\mu}\right|\der\mu\\
		=&\int_A \left(\der\mu_{\langle g_1;f\rangle}-\der\mu_{\langle g_2;f\rangle}\right)
		-\int_{X\setminus A} \left(\der\mu_{\langle g_1;f\rangle}-\der\mu_{\langle g_2;f\rangle}\right)\\
		\le& \left|\mu_{\langle g_1;f\rangle}(A)-\mu_{\langle g_2;f\rangle}(A)\right|
		+\left|\mu_{\langle g_1;f\rangle}(X\setminus A)-\mu_{\langle g_2;f\rangle}(X\setminus A)\right|\\
		\le& 2C_p\left(\Mc{E}(g_1)\vee \Mc{E}(g_2)\right)^{\frac{p-1-\alpha_p}{p}}
		\Mc{E}(g_1-g_2)^{\frac{\alpha_p}{p}}\Mc{E}(f)^{\frac{1}{p}},
	\end{align*}
	where $\alpha_p=\frac{1}{p}\wedge\frac{p-1}{p}$ and $C_p$ are the constants in \cite[Theorem 4.5]{KSdiff}.
	This implies that $g\mapsto \frac{\der\mu_{\langle g;f\rangle}}{\der\mu}$ is continuous. The continuity of
	$g\mapsto \frac{\der\mu_{\langle f;g\rangle}}{\der\mu}$ follows in the same way.
\end{proof}

\begin{thm}[Chain rule]\label{thmchain}
	Let $f,g\in\core$, and let $\varphi\in C(\Mb{R})$ be a piecewise $C^1$ function with $\varphi(0)=0$.
	Then $\varphi\circ f\in\core$ and
	\[
	\frac{\der \mu_{\langle \varphi\circ f;g \rangle}}{\der\mu}
	=\sgn(\varphi'\circ f)\,|\varphi'\circ f|^{p-1}\,\frac{\der\mu_{\langle f;g\rangle}}{\der\mu},
	\quad
	\frac{\der\mu_{\langle g;\varphi\circ f\rangle}}{\der\mu}
	=(\varphi'\circ f)\,\frac{\der\mu_{\langle g;f\rangle}}{\der\mu}
	\quad \mu\text{-a.e.},
	\]
	where we set $\varphi'\circ f(x):=0$ for $x\in f^{-1}\!\left(\{t\in\Mb{R}\mid \varphi'(t)\text{ does not exist}\}\right)$.
\end{thm}

\begin{proof}
	Note that $\varphi\circ f\in\core$ by Corollary \ref{corincl}~(3).
	We prove only the first identity; the second one is proved similarly.
	
	Assume first that $\varphi$ is piecewise linear. Choose $a_0,\dots,a_n\in\Mb{R}$ such that
	$-\|f\|_\infty=a_0<a_1<\cdots<a_n=\|f\|_\infty$ and $\varphi|_{[a_{i-1},a_i]}$ is affine for $1\le i\le n$.
	Then for any $A\in\Mc{B}(X)$,
	\begin{itemize}
		\item $\mu_{\left\langle \varphi\circ f;g\right\rangle }(A\cap f^{-1}(\{a_i\}))=\mu_{\left\langle f;g\right\rangle}(A\cap f^{-1}(\{a_i\}))=0$,
		\item $\begin{multlined}[t][.92\textwidth]
			\mu_{\left\langle \varphi\circ f;g\right\rangle}(A\cap f^{-1}((a_{i-1},a_i)))=\mu_{\left\langle (\varphi'\circ f)f;g\right\rangle}(A\cap f^{-1}((a_{i-1},a_i)))\\
			=\sgn(\varphi'\circ f)|\varphi'\circ f|^{p-1}\mu_{\left\langle f;g\right\rangle}(A\cap f^{-1}((a_{i-1},a_i))),
		\end{multlined}$
	\end{itemize}
	by Lemma \ref{lemsmwloc} and \cite[Theorem 4.5]{KSdiff}. These prove the claim in this case.\par
	
	Next let $\varphi$ be general. We may assume
	$\LIP(\varphi)=\LIP(\varphi|_{[-\|f\|_\infty,\|f\|_\infty]})<\infty$ without loss of generality.
	Let $\{\varphi_n\}_{n\ge 1}$ be a sequence of piecewise linear functions such that
	$\varphi_n(a)=\varphi(a)$ for every $a\in 2^{-n}\Mb{Z}$.
	Then $\varphi_n\to\varphi$ pointwise on $\Mb{R}$, $\LIP(\varphi_n)\le \LIP(\varphi)$, and
	$\varphi_n'\to \varphi'$ pointwise on
	\[
	\Mb{R}\setminus\Bigl(\bigcup_{n\ge1}2^{-n}\Mb{Z}\ \cup\ \{t\in\Mb{R}\mid \varphi'(t)\text{ does not exist}\}\Bigr),
	\]
	by the mean value theorem.
	Since $\varphi_n\circ f\in\core$ and
	\[
	\|\varphi_n\circ f\|_{\Mc{E}_1}
	\le \LIP(\varphi)\left(\Mc{E}(f)+\|f\|_\infty^p\,\Mf{m}(\supp(f))\right)^{1/p}<\infty
	\]
	(note that $|\varphi_n\circ f|\le \LIP(\varphi)\,|f|$),
	by \cite[Theorem 2.4.1]{HKST} and the Mazur lemma (see, e.g., \cite[Section 2.3]{HKST})
	there exist sequences $\{n_k\}_{k\ge1}$, $\{N_k\}_{k\ge1}$, and $\{\lambda_{k,i}\}_{k\ge1,\,n_k\le i\le N_k}$ such that
	$
	k\le n_k\le N_k,\ 0\le \lambda_{k,i},\ \sum_{i=n_k}^{N_k}\lambda_{k,i}=1$,
	and 
	$\psi_k\circ f$
	converge in $\Mc{E}_1$-norm, fore the convex combination $
	\psi_k:=\sum_{i=n_k}^{N_k}\lambda_{k,i}\,\varphi_i
	$.
	Let $h$ be the $\Mc{E}_1$-limit of $\{\psi_k\circ f\}_{k\ge1}$. Taking a subsequence if necessary, we may assume
	$\psi_k\circ f\to h$ $\Mf{m}$-a.e. Since $\varphi_n\circ f\to \varphi\circ f$ pointwise, we conclude that $h=\varphi\circ f$.
	
	Therefore, by Lemma \ref{lemconti},
	\[
	\frac{\der \mu_{\langle \varphi\circ f;g\rangle}}{\der\mu}
	=\lim_{k\to\infty}\frac{\der\mu_{\langle \psi_k\circ f;g\rangle}}{\der\mu}
	=\lim_{k\to\infty}\sgn(\psi_k'\circ f)\,|\psi_k'\circ f|^{p-1}\,\frac{\der \mu_{\langle f;g\rangle}}{\der \mu}
	\quad\text{in }L^1(\mu).
	\]
	Moreover,
	\[
	\frac{\der \mu_{\langle f;g\rangle}}{\der \mu}=0\quad \mu\text{-a.e.\ on }
	f^{-1}\!\left(\bigcup_{n\ge 1}2^{-n}\Mb{Z}\ \cup\ \{t\in\Mb{R}\mid \varphi'(t)\text{ does not exist}\}\right)
	\]
	by Lemma \ref{lemsmwloc}~(1). Hence
	\begin{align}\label{eqchainconv2}
		\lim_{k\to\infty}\sgn(\psi_k'\circ f)\,|\psi_k'\circ f|^{p-1}\,\frac{\der \mu_{\langle f;g\rangle}}{\der \mu}
		&=\lim_{k\to\infty}\sgn(\varphi_k'\circ f)\,|\varphi_k'\circ f|^{p-1}\,\frac{\der \mu_{\langle f;g\rangle}}{\der \mu}\notag\\
		&=\sgn(\varphi'\circ f)\,|\varphi'\circ f|^{p-1}\,\frac{\der \mu_{\langle f;g\rangle}}{\der \mu}
		\qquad \mu\text{-a.e.}
	\end{align}
	Since $|\psi_k'\circ f|^{p-1}\le \LIP(\varphi)^{p-1}\,\indi_{\supp(f)}\in L^1(\mu)$,
	\eqref{eqchainconv2} also holds in $L^1(\mu)$ by the dominated convergence theorem.
	This proves
	\[
	\frac{\der  \mu_{\langle \varphi\circ f;g \rangle}}{\der \mu}
	=\sgn(\varphi'\circ f)\,|\varphi'\circ f|^{p-1}\,\frac{\der \mu_{\langle f;g\rangle}}{\der \mu}
	\quad \mu\text{-a.e.}
	\]
\end{proof}

\begin{rem}
	The proof of Theorem \ref{thmchain} is inspired by \cite[Theorem 4.1.4]{GP} (see also \cite[Theorem 2.2.6]{Gig18}; cf.\ \cite[Theorems 4.15 and 5.11]{KSdiff}).
	In contrast to \cite[Proposition 4.1]{BV}, \cite[Theorem 7.4]{Shi}, and \cite[Theorem 8.17]{MS}, this proof does not require condition (H6) in \cite{BV}, which corresponds to \ref{defMarkov} with $\Mc{E}(\cdot)$ replaced by $\mu_{\langle \cdot\rangle}(A)$ $(A\in\Mc{B}(X))$.
\end{rem}

\begin{cor}\label{corLeibniz}
	For any $f,g,h\in\core$,
	\[
	\frac{\der \mu_{\langle f;gh\rangle}}{\der \mu}
	=g\,\frac{\der \mu_{\langle f;h \rangle}}{\der \mu}
	+h\,\frac{\der \mu_{\langle f;g\rangle}}{\der \mu}
	\quad \mu\text{-a.e.}
	\]
\end{cor}
\begin{proof}
	Applying Theorem \ref{thmchain} to $\varphi(t)=t^2$, we obtain
	\begin{align*}
		\frac{1}{2}\frac{\der\mu_{\langle f;(g+h)^2\rangle}}{\der\mu}
		&=\frac{1}{2}\left(
		\frac{\der\mu_{\langle f;g^2\rangle}}{\der\mu}
		+\frac{\der\mu_{\langle f;2gh\rangle}}{\der\mu}
		+\frac{\der\mu_{\langle f;h^2\rangle}}{\der\mu}
		\right)\\
		&=g\,\frac{\der\mu_{\langle f;g\rangle}}{\der\mu}
		+\frac{\der\mu_{\langle f;gh\rangle}}{\der\mu}
		+h\,\frac{\der\mu_{\langle f;h\rangle}}{\der\mu},
		\shortintertext{and}
		\frac{1}{2}\frac{\der\mu_{\langle f;(g+h)^2\rangle}}{\der\mu}
		&=(g+h)\,\frac{\der\mu_{\langle f;g+h\rangle}}{\der\mu}\\
		&=g\,\frac{\der\mu_{\langle f;g\rangle}}{\der\mu}
		+g\,\frac{\der\mu_{\langle f;h\rangle}}{\der\mu}
		+h\,\frac{\der\mu_{\langle f;g\rangle}}{\der\mu}
		+h\,\frac{\der\mu_{\langle f;h\rangle}}{\der\mu}
	\end{align*}
	$\mu$-a.e. Comparing these two identities yields the desired formula.
\end{proof}

\begin{proof}[Proof of Theorem \ref{thmunique}]
	The existence of energy measures follow from Corollaries \ref{corexist} and \ref{corLeibniz}, together with Theorem \ref{thmchain}. Theorem \ref{thmdomination} leads to the uniqueness. (Note that conditions \ref{condmeaspm} and \ref{condmeaslim} immediately follow from \ref{condmeasnorm} together with \ref{condmeasene}.)
\end{proof}


For the remainder of this section, we prove some fundamental properties of energy measures that follow from Theorems \ref{thmunique} and \ref{thmchain}.
The first is an identity corresponding to \eqref{defifunct4em}.

\begin{cor}\label{corfunct}
	For any $f,g\in\core$,
	\[
	g\,\frac{\der \mu_{\langle f\rangle}}{\der \mu}
	=\frac{\der \mu_{\langle f;fg\rangle}}{\der \mu}
	-\left(\frac{p-1}{p}\right)^{p-1}\frac{\der \mu_{\bigl\langle |f|^{\frac{p}{p-1}}; g\bigr\rangle}}{\der \mu}
	\quad \mu\text{-a.e.}
	\]
	In particular,
	\begin{equation}\label{eqfunctional}
		\int_X g\,\der\mu_{\langle f\rangle}
		=\Mc{E}(f;fg)-\left(\frac{p-1}{p}\right)^{p-1}\Mc{E}\left(|f|^{\frac{p}{p-1}};g \right),
	\end{equation}
	where
	\[
	\Mc{E}(f;g)
	=\mu_{\langle f;g\rangle}(X)
	=\frac{1}{p}\frac{\der}{\der t}\mu_{\langle f+tg\rangle}(X)\Big|_{t=0}
	=\frac{1}{p}\frac{\der}{\der t}\Mc{E}(f+tg)\Big|_{t=0}.
	\]
\end{cor}
\begin{proof}
	This follows from Theorem \ref{thmchain} and Corollary \ref{corLeibniz} (cf.\ \cite[Proposition 4.16]{KSdiff}).
\end{proof}

Similarly to the case of Dirichlet forms (cf.\ \cite[Theorem 3.2.2]{FOT}), the second identity in Theorem \ref{thmchain} extends to the case where $\varphi$ is an $n$-variable function.

\begin{thm}\label{thmchain2}
	Let $f, g_1,\dots g_n\in\core$ and let $\varphi\in C^1(\Mb{R}^n)$ satisfy $\varphi(\zero)=0$.
	Set $\Mf{g}=(g_1,\dots,g_n)$. Then
	\begin{equation}\label{eqnchain}
		\frac{\der \mu_{\langle f;\varphi\circ \Mf{g}\rangle}}{\der \mu}
		=\sum_{i=1}^n \left(\frac{\partial \varphi}{\partial x_i}\circ \Mf{g}\right)\frac{\der \mu_{\langle f;g_i\rangle}}{\der \mu}
		\quad \mu\text{-a.e.}
	\end{equation}
\end{thm}
\begin{proof}
	Fix $f,g_1,\dots,g_n\in\core$ and let $\Mc{A}\subset C^1(\Mb{R}^n)$ be the set of functions $\varphi$ such that $\varphi(\zero)=0$ and \eqref{eqnchain} holds.
	Then $\Mc{A}$ is an $\Mb{R}$-algebra by \eqref{eqemlinear}, Theorem \ref{thmchain}, and Corollary \ref{corLeibniz}.
	In particular, $\Mc{A}$ contains all polynomials $\varphi$ with $\varphi(\zero)=0$.
	
	Fix $\varphi\in C^1(\Mb{R}^n)$ with $\varphi(\zero)=0$.
	By \cite[Chapter II, Section 4.3]{CH}, there exist polynomials $\varphi_m$ $(m\ge 1)$ with $\varphi_m(\zero)=0$ such that
	$\varphi_m\to \varphi$ and $\frac{\partial \varphi_m}{\partial x_i}\to\frac{\partial \varphi}{\partial x_i}$ $(1\le i\le n)$
	uniformly on $[-\|\Mf{g}\|_{\infty},\|\Mf{g}\|_{\infty}]^n$.
	Arguing as in the proof of Theorem \ref{thmchain} (using convex combinations of a subsequence if necessary),
	we may assume $\varphi_m\circ \Mf{g}\to \varphi\circ \Mf{g}$ in $(\Mc{F},\|\cdot\|_{\Mc{E}_1})$ by Mazur's lemma.
	Hence, by Lemma \ref{lemconti} and the dominated convergence theorem,
	\begin{multline*}
		\frac{\der \mu_{\langle f;\varphi\circ \Mf{g}\rangle}}{\der \mu}
		=\lim_{m\to\infty}\frac{\der \mu_{\langle f;\varphi_m\circ \Mf{g}\rangle}}{\der \mu}
		\\=\lim_{m\to\infty}\sum_{i=1}^n \left(\frac{\partial \varphi_m}{\partial x_i}\circ \Mf{g}\right)\frac{\der \mu_{\langle f;g_i\rangle}}{\der \mu}
		=\sum_{i=1}^n \left(\frac{\partial \varphi}{\partial x_i}\circ \Mf{g}\right)\frac{\der \mu_{\langle f;g_i\rangle}}{\der \mu}
	\end{multline*}
	in $L^1(\mu)$.
\end{proof}

Theorem \ref{thmchain} admits another extension with respect to $\varphi$.

\begin{thm}\label{thmchain3}
	Theorem \ref{thmchain} remain valid if we replace ``piecewise $C^1$'' by ``locally Lipschitz''.
\end{thm}

To prove this theorem, we need a property of energy measures, called the \emph{energy image density property} in the case of Dirichlet forms (cf.\ \cite[Chapter 1, Theorem 7.1.1]{BH}), and also called the \emph{image density property} for $p$-energy measures (cf.\ \cite[Theorem 4.17]{KSdiff}).

\begin{thm}[{Special case of \cite[Theorem 4.17]{KSdiff}}]\label{thmeidm}
	If $f\in\core$, then $f_*\mu_{\langle f\rangle}\ll \Mc{L}(\Mb{R})$, where $\Mc{L}$ denotes Lebesgue measure.
\end{thm}

\begin{proof}[Proof of Theorem \ref{thmchain3}.]
	Note that $\Mc{L}(\{t\in\Mb{R}\mid \varphi'(t)\text{ does not exist}\})=0$ by Rademacher's theorem.
	Since $\|f\|_{\infty}<\infty$, we may assume $\varphi\in\LIP(\Mb{R})$ without loss of generality.
	Then there exist piecewise linear functions $\varphi_n$ $(n\ge 1)$ such that
	$
	\varphi_n(0)=0,\ \LIP(\varphi_n)\le \LIP(\varphi),\ \varphi_n\to\varphi \ \text{pointwise on }\Mb{R}$,
	and $\varphi_n'\to \varphi'$ $\Mc{L}$-a.e.\ (see, e.g., the proof of \cite[Theorem 4.1.4]{GP}).
	Therefore the same argument as in the proof of Theorem \ref{thmchain} applies, with
	\begin{gather*}
		\Bigl(\bigcup_{n\ge 1}2^{-n}\Mb{Z}\Bigr)\cup \{t\mid\varphi'(t)\text{ does not exist}\}
		\shortintertext{replaced by}
		\{t\in\Mb{R}\mid \lim_{n\to\infty}\varphi_n'(t)\ne \varphi'(t)\}\ \cup\ \{t\in\Mb{R}\mid \varphi'(t)\text{ does not exist}\},
	\end{gather*}
	using Theorem \ref{thmeidm}.
\end{proof}

\begin{rem}
	\begin{enumerate}
		\item For a further application of Theorem \ref{thmchain2}, see \cite[Theorem 4.15]{KSdiff}.
		\item Theorem \ref{thmchain3} does not extend to the case of $n$-variable functions in general.
		For example, if $\varphi(x_1,x_2):=x_1\vee x_2$, then
		$\frac{\der \mu_{\langle f;\varphi(g,g)\rangle}}{\der \mu}=\frac{\der \mu_{\langle f;g\rangle}}{\der \mu}$,
		but $\frac{\partial \varphi}{\partial x_i}\circ (g,g)$ $(i=1,2)$ is not defined on $X$.
	\end{enumerate}
\end{rem}

\section{\texorpdfstring{Korevaar--Schoen-type $p$-energy forms and canonical energy measures}{Korevaar--Schoen type p-energy forms and canonical energy measures}}\label{secKS}
In this section, we review the results of \cite{AB,KSKS} on $p$-energy measures associated with Korevaar--Schoen-type $p$-energy forms (named after Korevaar and Schoen \cite{KoSc}), and apply the results of the previous sections to this class.
Recall that we continue to assume Assumption \ref{asmspace}: in \cite{AB} the case $p=1$ is also considered, but we do not discuss it here.\par

We first introduce some notions following \cite{KSKS}.

\begin{defi}
	Let $\boldsymbol{k}=\{k_r\}_{r>0}$ be a family of $[0,\infty]$-valued Borel measurable functions on $X\times X$.
	For $f\in L^p(\Mf{m})$, define
	\begin{align*}
		J^{\boldsymbol{k}}_{p,r}(f)
		&:=\int_X\int_X |f(x)-f(y)|^p\,k_r(x,y)\,\der \Mf{m}(y)\,\der \Mf{m}(x),\\
		\Mf{B}^{\boldsymbol{k}}_{p,\infty}
		&:=\Bigl\{f\in L^p(\Mf{m}) \,\Big|\, \sup_{r>0}J^{\boldsymbol{k}}_{p,r}(f)<\infty \Bigr\},\\
		\|f\|_{\Mf{B}^{\boldsymbol{k}}_{p,\infty}}
		&:=\|f\|_{L^p(\Mf{m})}+\left(\sup_{r>0}J^{\boldsymbol{k}}_{p,r}(f)\right)^{1/p}
		\qquad \bigl(f\in\Mf{B}^{\boldsymbol{k}}_{p,\infty}\bigr).
	\end{align*}
	We say that $\boldsymbol{k}$ is \emph{asymptotically local} if there exists a map $\delta:(0,\infty)\to(0,\infty)$ such that
	$\lim_{r\downarrow 0}\delta(r)=0$ and
	\[
	\lim_{r\downarrow 0}\int_X\int_{X\setminus B_d(x,\delta(r))} k_r(x,y)\,\der \Mf{m}(y)\,\der \Mf{m}(x)=0.
	\]
	We say that \eqref{condWM} holds if there exists $C\ge 1$ such that
	\begin{equation}\label{condWM}
		\sup_{r>0}J^{\boldsymbol{k}}_{p,r}(f)\le C\liminf_{r\downarrow 0}J^{\boldsymbol{k}}_{p,r}(f)
		\tag{$\mathrm{WM}_{p}^{\boldsymbol{k}}$}
		\quad \text{for any }f\in\Mf{B}^{\boldsymbol{k}}_{p,\infty}.
	\end{equation}
\end{defi}

Condition \eqref{condWM} is called a \emph{weak monotonicity} type estimate; it was introduced in \cite{Bau}.
\begin{ex}
Assume that $(X,d,\Mf{m})$, $(\Mc{E},\Mc{F})$ and $\rho_p$ are as in Example \ref{exSC}. Let \[
	k_r(x,y)=\frac{\indi_{B_d(x,r)}(y)}{r^{\frac{\log (8\rho_p)}{\log 3}}\,\Mf{m}(B_d(x,r))},
	\qquad
	\boldsymbol{k}=\{k_r\}_{r>0}.
	\]
Then $\boldsymbol{k}$ is asymptotically local and \eqref{condWM} holds. Moreover, there exists $c>1$ such that $\Mc{F}=\Mf{B}^{\boldsymbol{k}}_{p,\infty}$, and for every $f\in\Mc{F}$, \[c^{-1}\sup_{r>0}J^{\boldsymbol{k}}_{p,r}(f)\le \Mc{E}(f)\le c\liminf_{r\downarrow 0}J^{\boldsymbol{k}}_{p,r}(f).\] See \cite{MS} for the proof and further details.
\end{ex}
\begin{thm}[{\cite[a part of Theorems 3.6, 3.8, 4.2, 4.5, 4.9, 4.11 and Proposition 4.7]{KSKS}; see also \cite[Proposition 2.3]{KSdiff}}]\label{thmKSKS}
	Suppose that $\boldsymbol{k}$ is asymptotically local and that \eqref{condWM} holds.
	Let $\{\tilde r_n\}_{n\ge 1}$ be any sequence of positive numbers with $\tilde r_n\to 0$ as $n\to\infty$.
	Then there exists a subsequence $\{r_n\}_{n\ge 1}$ such that, defining
	\[
	\Mc{E}_p^{\boldsymbol{k}}(f):=\lim_{n\to\infty}J^{\boldsymbol{k}}_{p,r_n}(f)\quad  (f\in \Mf{B}_{p,\infty}^{\boldsymbol{k}})
	\]
	the pair $(\Mc{E}_p^{\boldsymbol{k}},\Mf{B}^{\boldsymbol{k}}_{p,\infty})$ is a $p$-energy form satisfying
	\ref{defBanach}--\ref{defstrloc} and \ref{defclk}.
	In particular, if $\Mf{B}^{\boldsymbol{k}}_{p,\infty}\cap C_c(X)$ is dense in $(C_c(X),\|\cdot\|_\infty)$, then the closure
	\[
	\Mc{D}^{\boldsymbol{k},c}_{p,\infty}:=\ol{\Mf{B}^{\boldsymbol{k}}_{p,\infty}\cap C_c(X)}^{\|\cdot\|_{\Mf{B}^{\boldsymbol{k}}_{p,\infty}}},
	\]
	yields that  $(\Mc{E}_p^{\boldsymbol{k}},\Mc{D}^{\boldsymbol{k},c}_{p,\infty})$ is a $p$-energy form satisfying \ref{defBanach}--\ref{defclk}.
	In this case, there exists a unique family of Radon measures $\{\Mc{G}^{\boldsymbol{k}}_{p}\langle f\rangle\}_{f\in \Mc{D}^{\boldsymbol{k},c}_{p,\infty}}$ on $(X,d)$ satisfying \ref{condmeasene}--\ref{condmeasLeibniz} such that \eqref{eqfunctional} holds for any $f,g\in\Mf{B}^{\boldsymbol{k}}_{p,\infty}\cap C_c(X)$.
\end{thm}

\begin{cor}\label{corKSem}
	The measure $\Mc{G}^{\boldsymbol{k}}_{p}\langle f\rangle$ in Theorem \ref{thmKSKS} coincides with the canonical $p$-energy measure of $f$ associated with $(\Mc{E}_p^{\boldsymbol{k}},\Mc{D}^{\boldsymbol{k},c}_{p,\infty})$.
\end{cor}
\begin{proof}
	This is immediate from Theorems \ref{thmunique} and \ref{thmKSKS}.
\end{proof}

Kajino and Shimizu~\cite{KSKS} established detailed properties of these energy forms and the associated energy measures, in particular the chain and Leibniz rules, in this setting.
In \cite{KSKS}, the positivity of the right-hand side of \eqref{eqfunctional} is proved directly, and the $p$-energy measures are defined via the Riesz--Markov--Kakutani representation theorem, as in the case $p=2$.
In contrast, Alonso-Ruiz and Baudoin~\cite{AB} constructed $p$-energy forms and the associated energy measures in a more abstract way for PI spaces.

\begin{defi}[PI space]\label{defPIsp}
	We call $(X,d,\Mf{m})$ a \emph{($p$-)PI space} (or a \emph{($p$-)Cheeger space}) if it satisfies the volume doubling property and the $(p,p)$-Poincar\'e inequality.
	That is, there exist $C,\lambda>1$ such that
	\[
	0<\Mf{m}(B_d(x,2r))<C\,\Mf{m}(B_d(x,r))<\infty
	\]
	and
	\[
	\left(\int_{B_d(x,r)}\bigl| f(y)-f_{B_d(x,r)} \bigr|^p\, \der\Mf{m}(y)\right)^{1/p}
	\le Cr\left(\int_{B_d(x,\lambda r)}(\Lip f)(y)^p\,\der\Mf{m}(y)\right)^{1/p}
	\]
	for any $x\in X$, $r>0$, and any locally Lipschitz function $f\in C(X)$, where
	\[
	f_{B_d(x,r)}:=\frac{\int_{B_d(x,r)} f(y)\,\der\Mf{m}(y)}{\Mf{m}(B_d(x,r))},
	\quad
	(\Lip f)(x):=\limsup_{r\downarrow 0}\sup_{y\in B_d(x,r)}\frac{|f(x)-f(y)|}{r}.
	\]
\end{defi}

\begin{defi}[$\Gamma$- and $\bar{\Gamma}$-convergences]
	Let $(Y,\Mf{O})$ be a first-countable space and let $(Z,\rho)$ be a locally compact, separable metric space.
	Write $\Mc{O}_\rho:=\{U\subset Z \mid U \text{ is open}\}$.
	\begin{enumerate}
		\item Let $F,F_n:Y\to[0,\infty]$ $(n\ge1)$. We say that $\{F_n\}_{n\ge 1}$ \emph{$\Gamma$-converges} to $F$, and write $F=\Gamma\text{-}\lim_{n\to\infty}F_n$, if the following conditions hold:
		\begin{description}
			\Gcond{gammaliminf} For any convergent sequence $\{y_n\}_{n\ge 1}$ in $Y$,
			\[
			F\Bigl(\lim_{n\to\infty}y_n\Bigr)\le \liminf_{n\to\infty} F_n(y_n).
			\]
			\Gcond{gammarecovery} For any $y\in Y$, there exists a sequence $\{y_n\}_{n\ge 1}$ in $Y$ such that
			\[
			y=\lim_{n\to\infty} y_n
			\quad \text{and}\quad
			F(y)\ge\limsup_{n\to\infty}F_n(y_n).
			\]
		\end{description}
		Such a sequence $\{y_n\}_{n\ge1}$ is called a \emph{recovery sequence} of $\{F_n\}_{n\ge1}$ at $y$.
		
		\item Let $F,F_n:Y\times\Mc{O}_\rho\to[0,\infty]$ $(n\ge1)$ be such that $F(y,U)\le F(y,V)$ and $F_n(y,U)\le F_n(y,V)$ whenever $U\subset V$.
		We say that $\{F_n\}_{n\ge 1}$ \emph{$\bar\Gamma$-converges} to $F$, and write $F=\bar\Gamma\text{-}\lim_{n\to\infty}F_n$, if the following conditions hold:
		\begin{description}
			\BGcond{bargammainner} For any $V\in\Mc{O}_\rho$ and $y\in Y$,
			$
			F(y,V)=\sup\{F(y,U)\mid U\Subset V \}$.
			\BGcond{bargammaliminf} For any $U\in \Mc{O}_\rho$ and any convergent sequence $\{y_n\}_{n\ge 1}$ in $Y$,
			\[
			F\Bigl(\lim_{n\to\infty}y_n,\, U\Bigr)\le \liminf_{n\to\infty} F_n(y_n,U).
			\]
			\BGcond{bargammarecovery} For any $y\in Y$ and any $U,V\in\Mc{O}\rho$ with $U\Subset V$, there exists a sequence $\{y_n\}_{n\ge 1}$ in $Y$ such that
			\[
			y=\lim_{n\to\infty} y_n
			\quad \text{and}\quad
			F(y,V)\ge\limsup_{n\to\infty}F_n(y_n,U).
			\]
		\end{description}
		We call such a sequence $\{y_n\}_{n\ge1}$ a \emph{quasi-recovery sequence} of $\{F_n\}_{n\ge1}$ at $(y,V)$ with respect to $U$.
	\end{enumerate}
\end{defi}

\begin{rem}
	\begin{enumerate}
		\item This characterization is based on \cite[Proposition 8.1 and Remark 16.5]{Dal}. In \cite{Dal}, only the case where $Y$ is an open subset of $\Mb{R}^n$ is treated. However, the results from \cite{Dal} that we implicitly use extend to the present setting.
		\item Condition \ref{bargammainner} is necessary, although it is not assumed in \cite[Definition 2.9]{AB}. Indeed, let
		\[
		\Phi_a(U):=\begin{cases}
			0 & \ol{U}\text{ is compact},\\
			a & \text{otherwise},
		\end{cases}
		\]
		for $a\ge0$ and $U\in\Mc{O}_\rho$. Define $F_n(y,U):=\Phi_1(U)$ (independent of $y$ and $n$), and $F(y,U):=\Phi_a(U)$.
		Then $F$ satisfies \ref{bargammaliminf} and \ref{bargammarecovery} for any $a\in[0,1]$, but \ref{bargammainner} holds only when $a=0$.
		Since $F(\cdot,U)=\Phi_1(U)$ (as a constant function) is also the $\Gamma$-limit of $F_n(\cdot,U)$, this shows that set-wise $\Gamma$-limits need not coincide with the $\bar\Gamma$-limit.
	\end{enumerate}
\end{rem}

\begin{thm}[{\cite[Theorems 3.1 and 4.5]{AB}}]\label{thmAB}
	Assume that $(X,d,\Mf{m})$ is a PI space.
	For any $U\in\Mc{O}_d$, set
	\[
	k^{(U)}_r(x,y)=\frac{\indi_{B_d(x,r)}(y)\,\indi_U(x)}{r^p\,\Mf{m}(B_d(x,r))},
	\qquad
	\boldsymbol{k}^{(U)}=\{k^{(U)}_r \}_{r>0}.
	\]
	Then there exist a sequence $\{r_n\}_{n\ge 1}$ with $r_n\to 0$ as $n\to\infty$, a functional $\Mc{E}^\Gamma_p: L^p(\Mf{m})\to[0,\infty]$, and a set function $\Gamma_p: L^p(\Mf{m})\times\Mc{B}(X)\to[0,\infty]$ such that $\Mc{E}^\Gamma_p=\Gamma_p(\cdot,X)$,
	$J^{\boldsymbol{k}^{(X)}}_{p,r_n}$ $\Gamma$-converges to $\Mc{E}^\Gamma_p$, and
	$J^{\boldsymbol{k}^{(\cdot)}}_{p,r_n}(\cdot)$ $\bar\Gamma$-converges to $\Gamma_p|_{L^p(\Mf{m})\times\Mc{O}_d}$.
	Moreover,
	\[\Mr{KS}^{1,p}(X)
	:=\Mf{B}^{\boldsymbol{k}^{(X)}}_{p,\infty}
	=\left\{\limsup_{r\downarrow 0} J^{\boldsymbol{k}^{(X)}}_{p,r}<\infty\right\}
	=\left\{\Mc{E}^\Gamma_p<\infty\right\},\]
	and $\Gamma_p(f,\cdot)$ is a finite Radon measure for each $f\in \Mr{KS}^{1,p}(X)$.
	Furthermore, there exists $C>1$ such that for any $f\in\Mr{KS}^{1,p}(X)$ and any $U,V,W\in\Mc{O}_d$ with $U\Subset V\Subset W$,
	\begin{gather*}
		C^{-1}\sup_{r>0}J^{\boldsymbol{k}^{(X)}}_{p,r}(f)\le \Mc{E}^\Gamma_p(f)\le\liminf_{r\downarrow 0}J^{\boldsymbol{k}^{(X)}}_{p,r}(f),\\
		C^{-1}\limsup_{r\downarrow 0}J^{\boldsymbol{k}^{(U)}}_{p,r}(f)\le \Gamma_p(f,V)\le \liminf_{n\to\infty} J^{\boldsymbol{k}^{(V)}}_{p,r_n}(f)\le C\liminf_{r\downarrow 0}J^{\boldsymbol{k}^{(W)}}_{p,r}(f).
	\end{gather*}
\end{thm}

\begin{thm}\label{thmABem}
	The pair $(\Mc{E}^\Gamma_p,\Mr{KS}^{1,p}(X))$ is a $p$-energy form satisfying \ref{defBanach}--\ref{defclk}, and $\{\Gamma_p(f,\cdot)\}_{f\in\Mr{KS}^{1,p}(X)}$ is the family of canonical $p$-energy measures.
	In particular, \ref{condmeaschain} and \ref{condmeasLeibniz} hold with $\mu_{\langle f\rangle}=\Gamma_p(f,\cdot)$.
\end{thm}

For the proof of this theorem, we need a lemma that follows from the \emph{fundamental estimate} in \cite{AB} (introduced in \cite{Dal}).
To simplify notation, as in \cite{AB} we write
\[
E_{p,r}(f,U):=J^{\boldsymbol{k}^{(U)}}_{p,r}(f)
=\frac{1}{r^p}\int_U\int_{B_d(x,r)}\frac{|f(x)-f(y)|^p}{\Mf{m}(B_d(x,r))}\,\der\Mf{m}(y)\,\der\Mf{m}(x).
\]

\begin{lem}\label{lemKSloc}
	Let $f\in L^p(\Mf{m})\cap C_c(X)$, let $f_n\in L^p(\Mf{m})$ $(n\ge 1)$, and let $a\in\Mb{R}$ and $U\in\Mc{O}_d$.
	Assume that $f_n\to f$ in $L^p(\Mf{m})$ and that $\supp(f)\cap\supp(f+a)\Subset U$.
	Then
	\[
	\Mc{E}^\Gamma_p(f)\le \liminf_{n\to\infty}E_{p,r_n}(f_n,U).
	\]
\end{lem}
\begin{proof}
	Choose $V_1,V_2\in\Mc{O}_d$ such that
	$
	\supp(f)\cap\supp(f+a)\Subset V_1\Subset V_2\Subset U$.
	By \cite[Lemma 4.7]{AB}, there exist $C>0$ and $\Psi\in C(X)$ such that $0\le \Psi\le 1$, $\Psi|_{V_2}\equiv 1$, $\supp(\Psi)\subset U$, and
	\begin{align*}
		&E_{p,r_n}(\Psi f_n+(1-\Psi)f,X)-C\,\epsilon^{1-p}\|f_n-f\|^p_{L^p(\Mf{m})}\\
		\le&(1-\epsilon)^{1-p}\Bigl(E_{p,r_n}(f_n,U)+E_{p,r_n}(f, X\setminus\ol{V_1}) \Bigr)
	\end{align*}
	for any $\epsilon\in(0,1)$ and $n\ge 1$.
	Since $f\in C_c(X)$ implies $d(\{f=a\},\{f=0\})>0$ when $a\ne 0$, we have
	$E_{p,r_n}(f, X\setminus\ol{V_1})=0$ for all sufficiently large $n$.
	Taking $\liminf_{n\to\infty}$ and then letting $\epsilon\to 0$, we obtain
	\[
	\liminf_{n\to\infty}E_{p,r_n}(\Psi f_n+(1-\Psi)f,X)\le \liminf_{n\to\infty}E_{p,r_n}(f_n,U).
	\]
	Since $\Psi f_n+(1-\Psi)f\to f$ in $L^p(\Mf{m})$, the claim follows.
\end{proof}

\begin{proof}[Proof of Theorem \ref{thmABem}]
	By \cite[Proposition 4.11]{AB} (with $U=X$), the pair $(\Mc{E}^\Gamma_p,\Mr{KS}^{1,p}(X))$ is a $p$-energy form.
	By virtue of Theorem \ref{thmdomination} (cf.\ proof of Theorem \ref{thmunique}), it suffices to verify \ref{defBanach}--\ref{defclk} and \ref{condmeasene}--\ref{condmeasnorm}.
	\begin{description}
		\item[\ref{defBanach}] This follows from \cite[Theorem 3.1]{AB} and \cite[Theorem 3.1]{Bau}.
		\item[\ref{defMarkov}] This follows from \cite[Proposition 3.8]{AB}.
		
		\item[\ref{defstrloc}]
		Let $f,g\in\Mr{KS}^{1,p}(X)\cap C_c(X)$. Let $\{f_n\}_{n\ge 1}$ (resp. $\{g_n\}_{n\ge1},\{h_n\}_{n\ge 1}$) be a recovery sequence of $E_{p,r_n}(\cdot, X)$ at $f$ (resp. $g$, $f+g$).
		Take $U,U',V,V'\in\Mc{O}_d$ such that
		$
		\supp(f)\Subset U\Subset U'\subset\{g=a\},\
		\supp(g)\cap\supp(g+a)\Subset V\Subset V',\
		\ol{U'}\cap \ol{V'}=\emptyset$.
		Since
		\[
		\bigl(\indi_{U'}f_n+a\,\indi_{\{g=a\}\setminus V'}\bigr)+\indi_{V'}g_n \to f+g \quad \text{in }L^p(\Mf{m}),
		\]
		and $\Subset \supp(f+g)\cap\supp(f+g+a)\subset U\cup V$, we have
		\begin{align*}
			\Mc{E}^\Gamma_p(f)+\Mc{E}^\Gamma_p(g)
			&=\lim_{n\to\infty}\bigl(E_{p,r_n}(f_n,X)+E_{p,r_n}(g_n,X)\bigr)\\
			&\ge \liminf_{n\to\infty}\bigl(E_{p,r_n}(\indi_{U'}f_n, U)+E_{p,r_n}(\indi_{V'}g_n, V)\bigr)\\
			&= \liminf_{n\to\infty}\bigl(E_{p,r_n}(\indi_{U'}f_n+a\,\indi_{\{g=a\}\setminus V'}, U)+E_{p,r_n}(\indi_{V'}g_n, V)\bigr)\\
			&= \liminf_{n\to\infty}E_{p,r_n}\bigl(\indi_{U'}f_n+a\,\indi_{\{g=a\}\setminus V'}+\indi_{V'}g_n,\,U\cup V\bigr)\\
			&\ge \Mc{E}^\Gamma_p(f+g)
		\end{align*}
		by Lemma \ref{lemKSloc} and the definition of $E_{p,r}$.
		Similarly, we obtain
		\begin{align*}
			\Mc{E}^\Gamma_p(f+g)
			&=\lim_{n\to\infty}E_{p,r_n}(h_n,X)\\
			&\ge \liminf_{n\to\infty}E_{p,r_n}(h_n,U)+\liminf_{n\to\infty}E_{p,r_n}(h_n,V)\\
			&=\liminf_{n\to\infty}E_{p,r_n}(\indi_{U'}(h_n-a),U)+\liminf_{n\to\infty}E_{p,r_n}(\indi_{V'}h_n,V)\\
			&\ge \Mc{E}^\Gamma_p(f)+\Mc{E}^\Gamma_p(g).
		\end{align*}
		
		\item[\ref{defreg}]
		First note that $\Lip(f)\le \LIP(f|_{\supp(f)})\,\indi_{\supp(f)}$ for any $f$, and hence
		$\LIP(X)\cap C_c(X)\subset \Mr{KS}^{1,p}(X)\cap C_c(X)$ by \cite[Proposition 2.13]{AB}.
		For $f\in C_c(X)$, define $f_i(x):=\inf_{y\in X}\{f(y)+i\,d(x,y)\}\in \LIP(X)$.
		If $f$ is nonnegative, then $f_i\in C_c(X)$ and $f_i\to f$ as $i\to\infty$ in the uniform norm (see, e.g., \cite[Theorem 6.8]{Hei}).
		For general $f$, the functions $(f^+)_i-((-f)^+)_i\in \LIP(X)\cap C_c(X)$ converge uniformly to $f$.
		Together with \cite[Corollary 3.4]{AB}, this yields \ref{defreg}.
		
		\item[\ref{defclk}]
		Let $\tau_r$ be the measure on $X\times X$ defined by
		\[
		\der \tau_r=\frac{\indi_{B_d(x,r)}(y)}{r^p \Mf{m}(B_d(x,r))}\,\der (\Mf{m}\times \Mf{m})(x,y).
		\]
		Since $E_{p,r}(f,X)=\|f(x)-f(y)\|^p_{L^p(\tau_r)}$, the $p$-Clarkson inequalities for $E_{p,r}(f,X)^{1/p}$ hold on $\{E_{p,r}(f,X)<\infty\}$ by \cite[Theorem 2]{Clk}.
		It is known that $\Gamma$-limits preserve the $p$-Clarkson inequalities (see, e.g., \cite[Proof of Theorem 5.9]{Shi}); hence \ref{defclk} holds.
		
		\item[\ref{condmeasnorm}]
		This follows from \cite[Proposition 4.11]{AB} and the outer regularity of $\Gamma_p(f,\cdot)$.
		
		\item[\ref{condmeasene}]
		This is contained in \cite[Theorem 4.5]{AB}, but the proof is not written there; we include it for the reader's convenience.\par
		For $f\in L^p(\Mf{m})$, we have $\Gamma_p(f,X)\le \Mc{E}^\Gamma_p(f)$ by \cite[Remark 4.3]{AB}.
		To prove the reverse inequality, let $f\in\Mr{KS}^{1,p}(X)\cap C_c(X)$, and choose $U\in \Mc{O}_d$ such that $\supp(f)\Subset U\Subset X$.
		Let $\{f_n\}_{n\ge 1}$ be a quasi-recovery sequence of $\{E_{p,r_n}\}_{n\ge 1}$ at $(f,X)$ with respect to $U$.
		Then
		\[
		\Mc{E}^\Gamma_p(f)\le \liminf_{n\to\infty}E_{p,r_n}(f_n,U)\le \Gamma_p(f,X)
		\]
		by Lemma \ref{lemKSloc}.
		
		For general $f\in\Mr{KS}^{1,p}(X)$, there exists $\{f_n\}_{n\ge1}\subset \Mr{KS}^{1,p}(X)\cap C_c(X)$ such that
		$\lim_{n\to\infty}\Mc{E}^\Gamma_p(f-f_n)=0$ and $\limsup_{n\to\infty}\Gamma_p(f-f_n,X)=0$ by \ref{defreg}.
		Hence, by the triangle inequality in \ref{condmeasnorm},
		\[
		\Gamma_p(f,X)=\lim_{n\to\infty}\Gamma_p(f_n,X)=\lim_{n\to\infty}\Mc{E}^\Gamma_p(f_n)=\Mc{E}^\Gamma_p(f).
		\]
		
		\item[\ref{condmeasloc}]
		Since \ref{condmeasnorm} holds, it suffices to show that $\Gamma_p(f,\{f=c\})=0$ for any $f\in\Mr{KS}^{1,p}(X)\cap C_c(X)$ and $c\in\Mb{R}$.
		We first prove the case $c=0$.
		Let $\varphi_a:=\eta_{a,\infty}+\eta_{-\infty,-a}$ for $a>0$, where $\eta_I$ is defined in Section \ref{secunique}.
		Fix $f$ and $a$, and let $\{f_n\}_{n\ge 1}$ be a quasi-recovery sequence at $(f,\{f\ne 0\})$ with respect to $\{|f|>a\}$.
		Choose a subsequence $\{f_{n_k}\}_{k\ge 1}$ such that $f_{n_k}\to f$ $\Mf{m}$-a.e.\ and $|f_{n_k}|\le g$ for some $g\in L^p(\Mf{m})$ (cf.\ Proposition \ref{propnMarkov}).
		Since $\varphi_{2a}$ is a normal contraction, the dominated convergence theorem yields
		$
		\lim_{k\to\infty}\|\varphi_{2a}\circ f-\varphi_{2a}\circ f_{n_k}\|_{L^p(\Mf{m})}=0.
		$
		Therefore,
		\begin{align*}
			\Mc{E}^\Gamma_p(\varphi_{2a}\circ f)
			&\le \liminf_{k\to\infty}E_{p,r_{n_k}}(\varphi_{2a}\circ f_{n_k},\{|f|>a\})\\
			&\le \liminf_{k\to\infty}E_{p,r_{n_k}}(f_{n_k},\{|f|>a\})
			\le \Gamma_p(f,\{f\ne 0\})
		\end{align*}
		by Lemma \ref{lemKSloc} and the definition of $E_{p,r}$.
		Since $\lim_{n\to\infty}\Mc{E}^\Gamma_p(\varphi_{2^{1-n}}\circ f)=\Mc{E}^\Gamma_p(f)$ by the dominated convergence theorem and Lemma \ref{lemconvconv}, we obtain
		\[
		\Mc{E}^\Gamma_p(f)\ge \Gamma_p(f,X)\ge \Gamma_p(f,\{f\ne 0\})\ge \Mc{E}^\Gamma_p(f),\]
		and hence $
		\Gamma_p(f,\{f=0\})=0$.
		
		For $c\ne 0$, using the same argument as in \cite[Lemma 4.4]{AB}, we have, for all sufficiently large $n$,
		\[
		\Gamma_p\bigl(f,\{|f-c|<2^{-n}\}\bigr)
		=\Gamma_p\left(\left|f-\frac{c}{2}\right|-\frac{|c|}{2},\,\{|f-c|<2^{-n}\}\right).
		\]
		(Here we use that $(f-c)\indi_{\{|f-c|<2^{-n}\}}\in L^p(\Mf{m})$ when $|c|>2^{-n}$.)
		Therefore, by the case $c=0$ applied to $\left|f-\frac{c}{2}\right|-\frac{|c|}{2}$, we conclude that
		\[
		\Gamma_p(f,\{f=c\})
		=\Gamma_p\left(\left|f-\frac{c}{2}\right|-\frac{|c|}{2},\,\{f=c\}\right)=0.
		\]
	\end{description}
\end{proof}

\begin{rem}
	The uniqueness of $\Gamma_p$ with respect to $\Mc{E}^\Gamma_p$, and the chain and Leibniz rules for $\Gamma_p$, are new results in this setting.
	Since the assumptions of Theorem \ref{thmKSKS} are implied by those of Theorem \ref{thmAB}, we have
	$\Mc{G}^{\boldsymbol{k}}_p\langle f\rangle=\Gamma_p(f,\cdot)$ whenever $\Mc{E}^{\boldsymbol{k}}_p=\Mc{E}^\Gamma_p$, by Corollary \ref{corKSem} and Theorem \ref{thmABem}.
	However, it is unknown in general whether either $\Mc{E}^{\boldsymbol{k}}_p$ or $\Mc{E}^\Gamma_p$ is independent of the choice of the sequence $\{r_n\}_{n\ge 1}$ (cf.\ \cite[Remark 3.5]{AB}).
\end{rem}
\section*{Acknowledgments}
The author would like to thank Naotaka Kajino and Ryosuke Shimizu for their helpful comments on an earlier version of this paper and for fruitful discussions. In particular, thanks to a comment by Kajino, the author was able to remove an additional assumption from a previous version. 
\section*{Funding sources}
This work was supported by JSPS KAKENHI (grant number JP24KJ0022) and the Research Institute for Mathematical Sciences, an International Joint Usage/Research Center at Kyoto University.

\section*{Declaration of generative AI and AI-assisted technologies in the manuscript preparation process} During the preparation of this work the author used ChatGPT (OpenAI) in order to proofread and refine the English writing, improve clarity and readability of mathematical exposition, and check for potential typographical and consistency issues in the manuscript, during the revision. After using this service, the author reviewed and edited the content as needed and take full responsibility for the content of the published article.
\end{document}